\RequirePackage{fix-cm}
\RequirePackage{amsmath}
\RequirePackage{amssymb}

\documentclass[12pt]{article}
\usepackage{mathrsfs}
\usepackage{bm}
\usepackage{amsthm}
\usepackage[colorlinks,linkcolor=blue,anchorcolor=red,citecolor=blue]{hyperref}
 \usepackage{geometry}
\newtheorem{Thm}{Theorem}[section]

\newtheorem{Lem}[Thm]{Lemma}
\newtheorem{Coro}[Thm]{Corollary}

\newcommand{\1}{\mathbf{1}}
\newcommand{\R}{\mathbb{R}}
\newcommand{\Rd}{{\mathbb{R}^d}}

\renewcommand{\Re}{\text{Re}}

\newcommand{\F}{\mathscr{F}}
\newcommand{\N}{\mathbb{N}}

\renewcommand{\L}{\mathscr{L}}

\usepackage{cite}

\allowdisplaybreaks

\renewcommand{\Re}{\text{Re}}

\renewcommand{\S}{\mathscr{S}}

\newcommand{\<}{\langle}
\renewcommand{\>}{\rangle}
\newcommand{\T}{\mathcal{T}}

\renewcommand{\H}{H(a^{1/2})}

\newcommand{\vertiii}[1]{{\left\vert\kern-0.25ex\left\vert\kern-0.25ex\left\vert #1 \right\vert\kern-0.25ex\right\vert\kern-0.25ex\right\vert}}

\title{Regularity of Non-cutoff Boltzmann Equation with Hard Potential}
\author{Dingqun DENG \thanks{email: dingqdeng2-c@my.cityu.edu.hk, Department of Mathematics, City University of Hong Kong, ORCID: 0000-0001-9678-314X } }

\begin{document}

\maketitle

\begin{abstract}
	This article proves the regularity for the Boltzmann equation without angular cutoff with hard potential. By sharpening the coercivity and upper bound estimate on the collision operator, analyzing the Poisson bracket between the transport operator and some weighted pseudo-differential operator, we prove the regularizing effect in space and velocity variables when the initial data has mild regularity. 
	
	{\bf Keywords:} Regularity, Smoothing effect, pseudo-differential calculus, Boltzmann equation without cutoff.
	
	\paragraph{Mathematics Subject Classification (2020)}
	35Q20  	Boltzmann equations, 
35B65  	Smoothness and regularity of solutions to PDEs, 
35D35  	Strong solutions to PDEs,
76P05  	Rarefied gas flows, Boltzmann equation in fluid mechanics.
\end{abstract}

%
%
%
%

\tableofcontents

\section{Introduction}

In this paper, we consider the regularity of Boltzmann equation in $d$-dimension, which describes the dynamics of dilute gas:
\begin{align}\label{eq1}
	F_t + v\cdot\nabla_x F = Q(F,F),
\end{align} 
where the unknown $F(x,v,t):\Rd\times\Rd\times[0,\infty)\to[0,\infty)$ represents the density of particles at time $t$, position $x\in\Rd$ and velocity $v\in\Rd$ with $d\ge 3$. The Boltzmann collision operator $Q(F,G)$ is a bilinear operator, describing the particle interactions, defined for sufficiently smooth functions $F,G$ by
\begin{align*}
	Q(F,G)(v) := \int_{\Rd}\int_{S^{d-1}}B(v-v_*,\sigma) (F'_*G' - F_*G) \, d\sigma dv_*,
\end{align*}
where $F' = F(x,v',t)$, $G'_* = G(x,v'_*,t)$, $F = F(x,v,t)$, $G_* =G(x,v_*,t)$,
and $(v,v_*)$ are the velocities of two gas particles before collision while
$(v',v'_*)$ are the velocities after collision satisfying the following conservation laws of momentum and energy,
\begin{align*}
	v+v_*=v'+v'_*,\ \ |v|^2+|v_*|^2=|v'|^2+|v'_*|^2.
\end{align*}As a consequence, for $\sigma\in \mathbb{S}^{d-1}$, the unit sphere in $\R^d$, we have the $\sigma$-representation:
\begin{align*}
	v' = \frac{v+v_*}{2}+\frac{|v-v_*|}{2}\sigma,\ \
	v'_* = \frac{v+v_*}{2}-\frac{|v-v_*|}{2}\sigma.
\end{align*}
Also we define the angle $\theta$ in the standard way
\begin{align*}
	\cos\theta = \frac{v-v_*}{|v-v_*|}\cdot \sigma,
\end{align*}where $\cdot$ denotes the inner product in $\R^d$.
\paragraph{Collision Kernel} The collision kernel cross section $B$ is defined as 
\begin{align*}
	B(v-v_*,\sigma) = |v-v_*|^\gamma b(\cos\theta),
\end{align*}
for some $\gamma\in\R$ and function $b$. Without loss of generality, we can assume $B(v-v_*,\sigma)$ is supported on $(v-v_*)\cdot\sigma\ge 0$, which corresponds to $\theta\in[0,\pi/2]$, since $B$ can be replaced by its symmetrized form $\overline{B}(v-v_*,\sigma) = B(v-v_*,\sigma)+B(v-v_*,-\sigma)$ in $Q(f,f)$.
Moreover, we are going to work on the collision kernel without angular cutoff, which corresponds to the case of inverse power interaction laws between particles. That is,
\begin{align}
	b(\cos\theta)\approx \theta^{-d+1-2s}\ \text{ on }\theta \in (0,\pi/2),
\end{align}
and
\begin{align}
	s\in (0,1), \quad \gamma\in (-d,\infty).
\end{align}For the Boltzmann equation without angular cutoff, the condition $\gamma+2s< 0$ is called soft potential while $\gamma+2s>0$ is called hard potential. For mathematical theory of Boltzmann equation, one may refer to \cite{Alexandre2001,Alexandre2009,Cercignani1994,Villani2002} for more introduction.
Our regularity results are restricted to the range of parameters $\gamma+2s>0$. 

\subsection{Preliminary Result}
In this paper, we write $L^2$ space to be the space on $x,v$, i.e. $L^2 := L^2(\R^d_x\times\R^d_v)$. Also we will denote the Sobolev norms for $k,m\in\R$:
\begin{align*}
	\|f\|_{H^k_vH^m_x}:&=\|\<D_x\>^m\<D_v\>^kf\|_{L^2_{x,v}},
\end{align*}where $\<D_x\>^mf=\F^{-1}_x(\<y\>^m\F_x f(y))$, $\<D_v\>^kf=\F^{-1}_v(\<\eta\>^m\F_v f(\eta))$.

We would like to apply the symbolic calculus in \cite{Global2019,Deng2020} for our study as the following. 
One may refer to the appendix as well as \cite{Lerner2010} for more information about pseudo-differential calculus. Let $\Gamma=|dv|^2+|d\eta|^2$ be an admissible metric. 
Define
\begin{align}\label{Defofa}
	a(v,\eta):=\<v\>^\gamma(1+|\eta|^2+|\eta\wedge v|^2+|v|^2)^s+K_0\<v\>^{\gamma+2s}
\end{align}to be a $\Gamma$-admissible weight, where $K_0>0$ is chosen as following and $|\eta\wedge v|=|\eta||v|\sin\theta_0$ with $\theta_0$ being the angle between $\eta$, $v$. 
Applying theorem 4.2 in \cite{Global2019} and lemma 2.1 and 2.2 in \cite{Deng2020a}, there exists $K_0>0$ such that the Weyl quantization $a^w:H(ac)\to H(c)$ and $(a^{1/2})^w:H(a^{1/2}c)\to H(c)$ are invertible, with $c$ being any $\Gamma$-admissible metric. The weighted Sobolev space $H(c)$ is defined by \eqref{sobolev_space}. The symbol $a$ is real and gives the formal self-adjointness of Weyl quantization $a^w$. By the invertibility of $(a^{1/2})^w$, we have equivalence 
\begin{align*}
	\|(a^{1/2})^w(\cdot)\|_{L^2_v}\approx\|\cdot\|_{\H_v},
\end{align*}as in appendix and hence we will equip $\H_v$ with norm $\|(a^{1/2})^w(\cdot)\|_{L^2_v}$.

We will study the Boltzmann equation \eqref{eq1} near the global Maxwellian equilibrium
\begin{align*}
	\mu(v)=(2\pi)^{-d/2}e^{-|v|^2/2}.
\end{align*}
So we set $F = \mu + \mu^{\frac{1}{2}}f$ and then the Boltzmann equation \eqref{eq1} becomes
\begin{align*}
	f_t + v\cdot\nabla_x f = Lf + \mu^{-1/2}Q(\mu^{1/2}f,\mu^{1/2}f),
\end{align*}
where $L$ is called the linearized Boltzmann operator given by
\begin{align*}
	Lf = \mu^{-1/2}Q(\mu,\mu^{1/2}f) + \mu^{-1/2}Q(\mu^{1/2}f,\mu) = L_1f+L_2f,
\end{align*}
where $L_1,L_2$ are defined as the following, by applying \eqref{Carleman},
\begin{align}
	L_1f &= \lim_{\varepsilon\to 0}\int_{\R^d,|h|\ge\varepsilon}\,dh\int_{E_{0,h}}\,d\alpha\,\tilde{b}(\alpha,h)\1_{|\alpha|\ge|h|}\frac{|\alpha+h|^{\gamma+1+2s}}{|h|^{d+2s}}\mu^{1/2}(v+\alpha-h)\notag\\
	&\qquad\qquad\Big((\mu^{1/2}(v+\alpha)f(v-h)-\mu^{1/2}(v+\alpha-h)f(v)\Big),\label{L1}\\
	L_2f &= \lim_{\varepsilon\to 0} \int_{\R^d,|h|\ge\varepsilon}\,dh\int_{E_{0,h}}\,d\alpha\,\tilde{b}(\alpha,h)\1_{|\alpha|\ge|h|}\frac{|\alpha+h|^{\gamma+1+2s}}{|h|^{d+2s}}\mu^{1/2}(v+\alpha-h)\notag\\
	\label{L2}	&\qquad\qquad\Big(\mu^{1/2}(v-h)f(v+\alpha)-\mu^{1/2}(v)f(v+\alpha-h)\Big),
\end{align}
which are well-defined for Schwartz function according to \cite{Global2019,Deng2020a}.
Here we use the principal value on $h$ in order to assure the integral are well-defined when the two terms in the parentheses is separated into two integral, where change of variable can be applied.
By section 3 in \cite{Deng2020a}, $L=L_1+L_2$ can be regarded as the standard pseudo-differential operator with symbols in $S(a)$.
Then by the unique extension of continuous operator, $L$ is a linear continuous operator from $H(ac)$ into $H(c)$ for any $\Gamma$-admissible weight function $c$. Also we define for multi-index $\alpha$ that 
\begin{equation*}
	\T(g,f,\partial^\alpha_v(\mu^{1/2})) = \int_{\Rd}dv_*\int_{S^{d-1}}d\sigma B(v-v_*,\sigma)\partial^\alpha_v(\mu^{1/2})(v_*)(g'_*f'-g_*f), \end{equation*}\begin{equation*}
	\Gamma(g,f) = \T(g,f,\mu).
\end{equation*}Then $\Gamma(f,f) = \mu^{-1/2}Q(\mu^{1/2}f,\mu^{1/2}f)$.

To better understanding the behavior of $\Gamma$ and $L$, we consider weighted Sobolev norm $\|(a^{1/2})^w(\cdot)\|_{L^2}$, triple norm $\vertiii{\cdot}$ in \cite{Alexandre2012} and the norm $|\cdot|_{N^{s,\gamma}}$ in \cite{Gressman2011},
where 
\begin{align*}
	\vertiii{f}^2:&=\int B(v-v_*,\sigma)\Big(\mu_*(f'-f)^2+f^2_*((\mu')^{1/2}-\mu^{1/2})^2\Big)\,d\sigma dv_*dv,\\
	|f|^2_{N^{s,\gamma}}:&=\|\<v\>^{\gamma/2+s}f\|^2_{L^2}+\int(\<v\>\<v'\>)^{\frac{\gamma+2s+1}{2}}\frac{(f'-f)^2}{d(v,v')^{d+2s}}\1_{d(v,v')\le 1},
\end{align*}with $d(v,v'):=\sqrt{|v-v'|^2+\frac{1}{4}(|v|^2-|v'|^2)^2}$. Then by (2.13)(2.15) in \cite{Gressman2011}, Proposition 2.1 in \cite{Alexandre2012} and Theorem 1.2 in \cite{Global2019}, for $f\in\S$, $l\in\R$, we have the equivalence of norms:
\begin{align}\label{equivalent_norm}
	\|(a^{1/2})^wf\|^2_{L^2_v}\approx\vertiii{f}^2\approx|f|^2_{N^{s,\gamma}}\approx (-Lf,f)_{L^2_v}+\|\<v\>^lf\|_{L^2_v},
\end{align}with constants depending on $l$. These norms essentially describe the behavior of Boltzmann collision operator.

Also it's necessary to describe the weight $\<v\>$. So we apply the norm $|f|_{N^{s,\gamma}_l}$ in \cite{Gressman2011} to estimate linearized Boltzmann operator $L$ and collision term $\Gamma$, where for hard potential $\gamma+2s\ge 0$, 
\begin{equation}
	|f|^2_{N^{s,\gamma}_l} = |\<v\>^{l+\gamma/2+s}f|^2_{L^2_v}+\int_{\Rd}dv\,\<v\>^{2l+\gamma+2s+1}\int_{\Rd}dv'\,\frac{(f'-f)^2}{d(v,v')^{d+2s}}\1_{d(v,v')\le 1},
\end{equation}
\begin{equation}
	\|f\|^2_{N^{s,\gamma}_l} = \int_{\Rd}|f|^2_{N^{s,\gamma}_l}\,dx.
\end{equation}
Notice that by Plancherel's Theorem, 
\begin{align}\label{widehat}
	\|f\|^2_{N^{s,\gamma}_l} = \|\widehat{f}\|^2_{N^{s,\gamma}_l},
\end{align} where $\widehat{f}$ is Fourier transform of $f$ with respect to $x$. 
The Lemma 2.6, (2.10) in \cite{Gressman2011} gives the following coercive inequality. We would like to use norm $\|\cdot\|_{N^{s,\gamma}_l}$ when $l>0$ and norm $\|(a^{1/2})^w(\cdot)\|_{L^2}$ when $l=0$. 
\begin{Lem}For any $l\ge 0$, there exists $\nu_0>0$, $C>0$ such that
	\begin{align}\label{estimateofL}
		\big(-\<v\>^{2l}Lg,g\big)_{L^2_v}\ge \nu_0 |g|_{N^{s,\gamma}_l}-C\|g\|^2_{L^2_v}.
	\end{align}
\end{Lem}
By (6.6) in \cite{Gressman2011}, we have the trilinear upper bound: for multi-index $\beta$, $l\ge 0$, 
\begin{align}\label{estiofGamma0}
	\big|\big(\<v\>^{2l}\T(g,f,\partial^\beta_v(\mu^{1/2})),h\big)_{L^2_v}\big|\le C_\alpha\|\<v\>^lg\|_{L^2_v}|f|_{N^{s,\gamma}_{l}}|h|_{N^{s,\gamma}_l}.
\end{align}

Theorem 1.3 in my previous paper \cite{Deng2020} proves the global existence of Boltzmann equation without angular cutoff for hard potential. We define 
\begin{align*}
	B= \overline{-v\cdot\nabla_x+L},
\end{align*}as the closure of $(B,H(a\<v\>\<y\>))$ in $L^2_{x,v}$. Then
\begin{Thm}[Theorem 1.3, \cite{Deng2020}]\label{Thm1}Suppose $d\ge3$, $m>\frac{d}{2}$, $\gamma+2s\ge 0$. There exists $\varepsilon_0>0$ so small that if 
	\begin{align}\label{initial}
		\|f_0\|_{X}\le \varepsilon_0,
	\end{align}where $X$ is defined as \begin{align*}
	(f,g)_X = \delta(f,g)_{L^2_vH^m_x} + \int^{\infty}_0(e^{\tau B}f,e^{\tau B}g)_{L^2_vH^m_x}\,d\tau,
\end{align*}  
then there exists an unique global weak solution $f$ to Boltzmann equation 
	\begin{align*}
		f_t=Bf+\Gamma(f,f),\quad f|_{t=0}=f_0,
	\end{align*} satisfying 
	\begin{align}\label{globaletimate}
		\|f\|_{L^\infty([0,\infty);L^2_vH^m_x)}+\|f\|_{L^2([0,\infty);\H H^m_x)}\le C\varepsilon_0,
	\end{align}with some constant $C>0$. 
\end{Thm}
By \cite{Deng2020}, the norm $X$ satisfies that for $p\in[1,\frac{2d}{d+2})$, 
\begin{align*}
\|f_0\|_{X} \lesssim	\|f_0\|_{L^2_vH^m_x}+\|(a^{-1/2})^wf_0\|_{L^2_v(L^p_x)}
\end{align*}

\subsection{Main Result}
Our main result is the regularity of Boltzmann equation without angular cutoff for hard potential. That is, the weak solution we obtain in Theorem \ref{Thm1} is actually smooth in spatial variable $x$ and velocity variable $v$. The previous works such as \cite{Chen2018,Alexandre2010,Alexandre2011aa} require the solution to have some regularity at the beginning, while our work require only the assumption on the initial data \eqref{initial1}. That is, we start from the weak solution \eqref{globaletimate} and find out its regularity directly. 

\begin{Thm}\label{Main}
	\label{Thm}Assume $\gamma+2s>0$ and $m_0>\frac{d}{2}$. Suppose $f_0\in L^2_vH^{m_0}_x$. There exists sufficiently small $\varepsilon_0>0$ such that if 
	\begin{align}
		\label{initial1}\|f_0\|_X\le \varepsilon_0,
	\end{align}then the weak solution $f$ to Boltzmann equation near Maxwellian obtained in Theorem \ref{Thm1}:
	\begin{align}\label{Boltzmann}
		\partial_tf = Bf + \Gamma(f,f),\quad f|_{t=0}=f_0,
	\end{align}satisfies that for $\tau>0$, $m,k,l\in\N$,  
\begin{align*}
	\|\<v\>^l\<D_v\>^k\<D_x\>^mf\|^2_{L^\infty([\tau,\infty);L^2_vH^{m_0}_x)}\le C_\tau\varepsilon_0^2.
\end{align*}
Consequently, 
$f(t)\in C^\infty(\R^d_x;\S(\R^d_v))$ for $t>0$.
\end{Thm}
In order to obtain the regularity on $x$, we follow the idea in \cite{Chen2018} and 
define 
\begin{align*}
b(v,y) &= (1+|v|^2+|y|^2+|v\wedge y|^2)^{\delta_1},\\
\theta(v,\eta) &= (1+|v|^2+|y|^2+|v\wedge y|^2)^{\delta_1-1} (y\cdot\eta+(v\wedge y)\cdot(v\wedge\eta))\chi(v,\eta).
\end{align*}
The Poisson bracket between $v\cdot\nabla_x$ and $\theta$ can give us the regularity of $f$ with order $b^{1/2}(v,y)$. In brief, this yields regularity on $x$ with order $\delta_1/2$. Noting the definition of $\delta_1$ in Theorem \ref{Thm33}, we require $\gamma+2s>0$ in order to obtain a positive order $\delta_1$, which is also the only place that require $\gamma+2s$ to be strictly positive.  
The assumption on initial data $f_0$ comes from Theorem \ref{Thm1}, which is natural for the existence theory of Boltzmann equation in the whole space in both cutoff and non cutoff cases.

The smoothing effect of the Boltzmann equation non-cutoff collision kernel were discussed in many context. At the beginning, entropy production estimate for non cutoff assumption were established, as in \cite{Alexandre2000,Lions1998}. Their result were widely applied in the theory of non-cutoff Boltzmann equation. Later on, many works discover the optimal regular estimate of Boltzmann collision operator in $v$ in different setting. We refer to \cite{Alexandre2011,Global2019,Alexandre2013,Gressman2011a,Mouhot2007} for the dissipation estimate of collision operator, and \cite{Alexandre2011aa,ALEXANDRE2005,Alexandre2009a,Alexandre2010,Barbaroux2017,Barbaroux2017a,Chen2018,Chen2011,Chen2012,Lekrine2009,Lerner2014,Lerner_2015} for smoothing effect of the solution to Boltzmann equation in different aspect. These works show that the Boltzmann operator behaves locally like a fractional operator:
\begin{align*}
	Q(f,g)\sim (-\Delta_v)^sg+\text{lower order terms}.
\end{align*}
More precisely, according to the symbolic calculus developed by Alexandre-H{\'{e}}rau-Li \cite{Global2019}, the linearized Boltzmann operator behaves essentially as 
\begin{align*}
	L \sim \<v\>^\gamma(-\Delta_v-|v\wedge\partial_v|^2+|v|^2)^s+\text{lower order terms}.
\end{align*}
This diffusion property shows that the spatially homogeneous Boltzmann equation behaves like fractional heat equation, while the spatially inhomogeneous Boltzmann behaves as the generalized Kolmogorov equation. We refer to \cite{Barbaroux2017a,Lekrine2009,Lerner_2015} for Kac equation, the one dimensional model of Boltzmann equation, and \cite{Morimoto2009} for similar kinetic equation.

\paragraph{Organization of the article}
Our analysis is organized as follows. In Section 2, we provide some basic lemmas that is applied from time to time in our analysis. In section 3, we give the rigorous argument for obtaining regularity on $x$ and $v$. The appendix gives some general theory on pseudodifferential calculus.

\paragraph{Notations}
Throughout this article, we shall use the following notations. 
$C$ represents the constants which might be changed from line to line. $\S(\Rd)$ is the set of Schwartz functions on $\Rd$.
For any $v\in\Rd$, we denote $\<v\>=(1+|v|^2)^{1/2}$. The gradient in $v$ is denoted by $\partial_v$. The notation $a\approx b$ (resp. $a\gtrsim b$, $a\lesssim b$) for positive real function $a$, $b$ means there exists $C>0$ not depending on possible free parameters such that $C^{-1}a\le b\le Ca$ (resp. $a\ge C^{-1}b$, $a\le Cb$) on their domain. $\Re (a)$ means the real part of complex number $a$. $[a,b]=ab-ba$ is the commutator between operators. $\{a(v,\eta),b(v,\eta)\} =  \partial_\eta a_1\partial_va_2 - \partial_va_1\partial_\eta a_2$ is the Poisson bracket. $\Gamma=|dv|^2+|d\eta|^2$ be the admissible metric and $S(m)=S(m,\Gamma)$ be the symbol class. 

For pseudo-differential calculus, we write $(x,v)\in \Rd\times\Rd$ is the space-velocity variable and $(y,\eta)\in \Rd\times\Rd$ is the corresponding variable in dual space (the variable after Fourier transform).

\section{Basic Lemmas}\label{sec2}
Assume $\gamma+2s>0$ and fix $m_0>\frac{d}{2}$. In this section, we give the sketch proof of our main theorem, the iteration of our solution. Also, some pseudo-differential calculus and the estimate on collision term $\Gamma(f,f)$ are discussed. 

\paragraph{Iteration}
As in \cite{Deng2020}, we will prove our main result by using iteration. That is, we set $f_0\in L^2_vH^{m_0}_x$ satisfying \eqref{initial1} and $f^0=0$. Then we solve linear equation
\begin{align}
	\partial_tf^{n+1} = Bf^{n+1} + \Gamma(f^n,f^{n+1}),\quad f^{n+1}|_{t=0} = f_{0},
\end{align}to get the sequence $\{f^n\}(n\in\N)$. 
Theorem 1.3 in \cite{Deng2020} shows that the sequence $\{f^n\}$ is strongly convergent to a limit $f$ in space $L^\infty_t([0,\infty),L^2H^{m_0}_x)$.

So in the followings, we will firstly analyze linear equation 
\begin{align}\label{lineareq}
	\partial_tf = Bf + \Gamma(g,f),\quad f|_{t=0} = f_{0}.
\end{align}We also set up the iteration assumption: Fix $\varepsilon_1\in (0,1]$ and for $m,k,l\in\N$, we assume
\begin{equation}
	\label{assumong}
	\begin{split}
		\|g\|_{L^2_vH^{m_0}_x} \le C\varepsilon_0,\\
		\sup_{0\le t\le T}\|t^{\kappa m}\<v\>^l\<D_v\>^k\<D_x\>^mg\|_{L^2_vH^{m_0}_x}^2\le \varepsilon_1,
	\end{split}
\end{equation}
In order to make the solution $f$ smooth enough to substitute as a test function, we consider the vanishing method. That is, we define for sufficiently large $N>\gamma+2s+2$ that
\begin{align*}
	M = \<v\>^N\<D_v\>^N\<D_x\>^N,
\end{align*}
and let $f=f_\varepsilon$ to be the solution of linear equation 
\begin{align}
	\partial_tf + v\cdot\nabla_xf + \varepsilon M^*Mf = Lf + \Gamma(g,f),\quad f|_{t=0} = f_{0,\varepsilon}, \label{lineareqvarepsilon}
\end{align}with mollified initial data $f_{0,\varepsilon}\in\S$, which converge to $f_0$ in $L^2_vH^{m_0}_x$. Then this $f=f_\varepsilon$ is smooth enough to become the test function.

\paragraph{Pseudodifferential Calculus}
Here we mainly give some control on invertible pseudodifferential operator, which is applied in our analysis from time to time.
Let $m_K(v,\eta)$ be a $\Gamma$-admissible weight function depending on $K$, $c$ be any $\Gamma$-admissible weight. Then lemma 2.1 and 2.3 in \cite{Deng2020a} can be reformulated as the following.
\begin{Lem}[Lemma 2.1, \cite{Deng2020a}]\label{inverse_pseudo}
	Assume $a_K\in S(m_K)$, $\partial_\eta (a_{K})\in S(K^{-\kappa}m_{K})$ uniformly in $K$ and $|a_{K}|\gtrsim m_{K}$. Then \\
	(1) $a^{-1}_{K}\in S(m^{-1}_{K})$, uniformly in $K$, for $K>1$.\\
	(2) There exists $K_0>1$ sufficiently large such that for all $K>K_0$, $a^w_{K}:H(m_Kc)\to H(c)$ is invertible and its inverse $(a^w_{K,l})^{-1}: H(c) \to H(m_Kc)$ satisfies
	\begin{align*}
		(a^w_{K})^{-1} = G_{1,K}(a^{-1}_{K})^w = (a^{-1}_{K})^wG_{2,K},
	\end{align*}where $G_{1,K}\in \L(H(m_Kc))$, $G_{2,K}\in \L(H(c))$ with operator norm smaller than $2$. Also, by the equivalence of invertibility, $(a_K^w)^{-1}\in Op(m^{-1}_K)$.
\end{Lem}
\begin{Lem}[Lemma 2.3, \cite{Deng2020a}]\label{pseudo1}Let $m,c$ be $\Gamma$-admissible weight and $A\in S(m)$. Assume $A^w:H(mc)\to H(c)$ is invertible.
	If $B\in S(m)$, then there exists $C>0$, depending only on the seminorms of symbols to $(A^w)^{-1}$ and $B^w$, such that for $f\in H(mc)$,
	\begin{align*}
		\|B(v,D_v)f\|_{H(c)}+\|B^w(v,D_v)f\|_{H(c)}\le C\|A^w(v,D_v)f\|_{H(c)}.
	\end{align*}
\end{Lem}
\begin{Coro}\label{pseudobound}Let $m,m_1,m_2$ be $\Gamma$-admissible weight. \\
(1)	If $A^w:H(m_1)\to L^2\in Op(m_1)$, $B^w:H(m_2)\to L^2\in Op(m_2)$ are invertible, then for $f\in\S$, 
	\begin{align*}
		\|B^wA^wf\|_{L^2}\approx \|A^wB^wf\|_{L^2},
	\end{align*}where the constant depends only on seminorms of symbols to $A^w,B^w,(A^w)^{-1},(B^w)^{-1}$. \\
(2) If $A^w:H(m)\to L^2\in Op(m)$, $B_1^w:H(m_1)\to L^2\in Op(m_1)$, $B_2^w:H(m_2)\to L^2\in Op(m_2)$ are invertible and $m\lesssim m_1+m_2$, then 
\begin{align*}
\|A^wf\|_{L^2}\le C(\|B_1^wf\|_{L^2}+\|B_2^wf\|_{L^2}) 
\end{align*}
\end{Coro}
\begin{proof}
	The first assertion is easily follows from Lemma \ref{pseudo1}. That is, for $f\in\S$, \begin{align*}
		\|B^wA^wf\|_{L^2} = \|\underbrace{B^wA^w(B^w)^{-1}}_{\in Op(m_1)}B^wf\|_{L^2} \lesssim \|A^wB^wf\|_{L^2}.
	\end{align*}The contrary is similar. 
For the second assertion, we use the weighted Sobolev space \eqref{sobolev_space}. Since $A^w$, $B_1^w$, $B_2^w$ are invertible, we can write $\|f\|_{H(m)} = \|A^wf\|_{L^2}$, $\|f\|_{H(m_1)} = \|B_1^wf\|_{L^2}$, $\|f\|_{H(m_2)} = \|B_2^wf\|_{L^2}$. Thus, by $m\lesssim m_1+m_2$, 
\begin{align*}
	\|A^wf\|_{L^2} &= \|f\|_{H(m)}\\
	&\lesssim \|f\|_{H(m_1)} + \|f\|_{H(m_2)}\\
	&\lesssim \|B_1^wf\|_{L^2} + \|B_2^wf\|_{L^2}.
\end{align*}
\end{proof}
For a direct application, we have that for $k,l\in\R$, $f\in\S$  
\begin{align}\label{eq80}
	\|(\<v\>^k\<\eta\>^l)^wf\|_{L^2_v} = \|\<v\>^k\<D_v\>^lf\|_{L^2}.
\end{align}

\paragraph{Estimate of $\Gamma(g,f)$}
Denote $\widehat{f}$ to be the Fourier transform of $f$ on $x$. By \eqref{estiofGamma0}, we have for $m>\frac{d}{2}$, $f,g,h\in\S$ that 
\begin{align*}
	&\quad\ \Big|\Big(\big(\<v\>^{2l}\<D_x\>^m\T(g,f,\partial^\alpha_v(\mu^{1/2})),\<D_x\>^mh\big)_{L^2_{x,v}}\Big|\\
	&= \Big|\Big(\<v\>^{2l}\<y\>^m\int\T(\widehat{g}(y-z),\widehat{f}(z),\partial^\alpha_v(\mu^{1/2}))\,dz,\<y\>^m\widehat{h}\Big)_{L^2_{y,v}}\Big|\\
	&\le  \int\int\<y\>^{m}\Big|\Big(\<v\>^{2l}\T(\widehat{g}(y-z),\widehat{f}(z),\partial^\alpha_v(\mu^{1/2})),\<y\>^{m}\widehat{h}(y)\Big)_{L^2_{v}}\Big|\,dzdy\\
	&\le C_m\int\int\|\<v\>^l\<y-z\>^m\widehat{g}(y-z)\|_{L^2_{v}}|\widehat{f}(z)|_{N^{s,\gamma}_l}|\<y\>^m\widehat{h}(y)|_{N^{s,\gamma}_l}\,dzdy\\
	&\qquad+ C_m\int\int\|\<v\>^l\widehat{g}(y-z)\|_{L^2_{v}}|\<z\>^m\widehat{f}(z)|_{N^{s,\gamma}_l}|\<y\>^m\widehat{h}(y)|_{N^{s,\gamma}_l}\,dzdy\\
	&\le C_m\|\<v\>^l\<D_x\>^mg\|_{L^2_{x,v}}\big\||\widehat{f}(y)|_{N^{s,\gamma}_l}\big\|_{L^1_y}\|\<D_x\>^m\widehat{h}\|_{N^{s,\gamma}_l}\\
	&\qquad + C_m\big\|\|\<v\>^l\widehat{g}(y)\|_{L^2_{v}}\big\|_{L^1_y}\|\<D_x\>^mf\|_{N^{s,\gamma}_l}\|\<D_x\>^m{h}\|_{N^{s,\gamma}_l},
\end{align*}by H\"{o}lder's inequality and Fubini's theorem. Notice that $\|\widehat{g}\|_{L^1_y}\le C\|\<y\>^m\widehat{g}\|_{L^2_y}$, since $m>d/2$. 
\begin{align}
	\Big|\big(\<v\>^{2l}&\T(g,f,\partial^\alpha_v(\mu^{1/2})),h\big)_{L^2_vH^m_x}\Big|
	\le C_m\|\<D_x\>^m\<v\>^lf\|_{L^2_{x,v}}\|\<D_x\>^mg\|_{N^{s,\gamma}_l}\|\<D_x\>^mh\|_{N^{s,\gamma}_l}.\label{estiofGamma}
\end{align}

\section{Regularity on $x,v$}
In this section, we will complete the iteration on linear equation \eqref{lineareq} and prove the regularity of the Boltzmann equation on $x$ and $v$. As the beginning, we will analyze the regularity estimate of mollified linear equation \eqref{lineareqvarepsilon}. 
\begin{Thm}\label{Thm31}Let $T\in(0,1]$, $m_0>\frac{d}{2}$. Assume $g$ satisfies \eqref{assumong}. Let $f$ be the solution to equation \eqref{lineareqvarepsilon}. We can choose $\varepsilon_0$ in \eqref{initial1} so small that 
	\begin{equation}\label{eq0}\begin{split}
			\sup_{0\le t\le T}\|f(t)\|_{L^2_vH^{m_0}_x}^2 +2 \varepsilon\int_{0}^{T}\|Mf\|^2_{L^2_vH^{m_0}_x}\,dt  + {\nu_0}\int_{0}^{T}\|(a^{1/2})^wf\|^2_{L^2_vH^{m_0}_x}\,dt
			&\le C\varepsilon_0^2,
		\end{split}
	\end{equation}
\end{Thm}
\begin{proof}
Let $t\in (0,T]$. We write $f$ to be the solution to equation \eqref{lineareqvarepsilon}. Then 
\begin{align*}
\big(\partial_tf+\varepsilon M^*Mf+v\cdot\nabla_xf,f\big)_{L^2_vH^{m_0}_x}
&= \big(Lf+\Gamma(g,f), f\big)_{L^2_vH^{m_0}_x}.
\end{align*}
Taking the real part, applying the coercive estimate \eqref{estimateofL} of $L$, the equivalence \eqref{equivalent_norm}, and estimate \eqref{estiofGamma} of $\Gamma$, 
\begin{align*}
&\quad\,\frac{1}{2}\partial_t\|f\|_{L^2_vH^{m_0}_x}^2 + \varepsilon\| Mf\|^2_{L^2_vH^{m_0}_x}  + \nu_0\|(a^{1/2})^wf\|^2_{L^2_vH^{m_0}_x}\\
&\le C\|f\|^2_{L^2_vH^{m_0}_x}+C\|g\|_{L^2_vH^{m_0}_x}\|(a^{1/2})^wf\|^2_{L^2_vH^{m_0}_x},
\end{align*}
Since $\sup_{0\le t\le T}\|g\|_{L^2_vH^{m_0}_x}<\sqrt{2}\varepsilon_0$, we choose $\varepsilon_0<<1$ that $\sqrt{2}C\varepsilon_0<\frac{\nu_0}{2}$. 
\begin{align*}
\frac{1}{2}\partial_t\|f\|_{L^2_vH^{m_0}_x}^2 + \varepsilon\| Mf\|^2_{L^2_vH^{m_0}_x}  + \frac{\nu_0}{2}\|(a^{1/2})^wf\|^2_{L^2_vH^{m_0}_x}
&\le C\|f\|^2_{L^2_vH^{m_0}_x}.
\end{align*}
Applying Gronwall's inequality, 
\begin{equation*}\begin{split}
&\quad\,\sup_{0\le t\le T}\|f(t)\|_{L^2_vH^{m_0}_x}^2 +2 \varepsilon\int_{0}^{T}\|Mf\|^2_{L^2_vH^{m_0}_x}\,dt  + {\nu_0}\int_{0}^{T}\|(a^{1/2})^wf\|^2_{L^2_vH^{m_0}_x}\,dt\\
&\le e^{2CT}\|f(0)\|^2_{L^2_vH^{m_0}_x}\\
&\le C\varepsilon_0^2,
	\end{split}
\end{equation*}where the constant $C$ is independent of time $T$, since $T\le 1$. 

\end{proof}

For higher order derivative, we will need to mollify $f$ with $t^{\kappa(m+k)}$ in order to eliminate the influence of initial data. So we have the following Theorem.

\begin{Thm}
Let $T\in(0,1]$, $m_0>\frac{d}{2}$. Assume $g$ satisfies \eqref{assumong}.	Let $f$ be the solution to equation \eqref{lineareqvarepsilon}.
	Let $m,k,l\in\N$ to be the index for derivative of $x,v$ and weight $\<v\>$ respectively. For sufficiently large $\kappa>1$ and small $\varepsilon_0$, the following estimate are valid. 
	For derivative of $v$, 
	\begin{align}\label{20}
		\notag&\quad\,\partial_t\|t^{\kappa k}f\|_{H^k_vH^{m_0}_x}^2 
		+ \varepsilon\|Mt^{\kappa k}f\|_{H^k_vH^{m_0}_x}^2
		+\nu_0\|(a^{1/2})^wt^{\kappa k}f\|_{H^k_vH^{m_0}_x}^2\\
		&\le C_{\kappa,k}\big(\varepsilon\|Mf\|^2_{L^2_vH^{m_0}_x} + \|(a^{1/2})^wf\|^2_{L^2_vH^{m_0}_x} + \|t^{\kappa k}\<D_x\>^kf\|^2_{L^2_vH^{m_0}_x}\big).
	\end{align}
For derivative of $x$, 
\begin{align}\label{21}
	\notag&\quad\,\partial_t\|t^{\kappa m}f\|_{L^2_vH^{m+m_0}_x}^2 
	+ \varepsilon\|Mt^{\kappa m}f\|_{L^2_vH^{m+m_0}_x}^2
	+\nu_0\|(a^{1/2})^wt^{\kappa m}f\|_{L^2_vH^{m+m_0}_x}^2\\
	&\le \delta\|(b^{1/2})^wt^{\kappa m}f\|^2_{L^2_vH^{m+m_0}_x} + C_{\delta}\|(a^{1/2})^wf\|^2_{L^2_vH^{m_0}_x}.
\end{align}
For the weight $\<v\>$, 
\begin{align}
	\label{24}
	\sup_{0\le t\le T}\|\<v\>^lf\|_{L^2_vH^{m_0}_x}^2 
	+ \varepsilon\int^T_0\|M\<v\>^lf\|_{L^2_vH^{m_0}_x}^2\,dt
	+\nu_0\int^T_0\|\<D_x\>^{m_0}f\|_{N^{s,\gamma}_l}^2\,dt
	&\le C\varepsilon_0^2.
\end{align}
\end{Thm}
\begin{proof}Take a large constant $\kappa>1$ to be chosen later. We mollify $f$ with $t^{\kappa (m+k)}$. Then 
\begin{align*}
\partial_t(t^{\kappa (m+k)} f) &= (t^{\kappa (m+k)})_tf+t^{\kappa (m+k)} f_t\\&
= (t^{\kappa (m+k)})_tf+t^{\kappa (m+k)}(-\varepsilon M^*Mf-v\cdot\nabla_xf+Lf+\Gamma(g,f)).
\end{align*}
Take muti-index $|\alpha|\le m$, $|\beta|\le k$. We use $(-\partial_x)^\alpha(-\partial_v)^\beta\<v\>^{2l}\<D_x\>^{2m_0}t^{\kappa (m+k)}\partial^\alpha_x\partial^\beta_v f$ as the test function. Notice that $\<D_x\>$ commutes with $M$, $v\cdot\nabla_x$ and $L$. Thus when $m+k>0$, 
\begin{equation}\label{eq4}\begin{split}
	&\quad\,\Big(\partial_t(t^{\kappa (m+k)}\<v\>^l\<D_x\>^{m_0}\partial^\alpha_x\partial^\beta_v f),t^{\kappa (m+k)}\<v\>^l\<D_x\>^{m_0}\partial^\alpha_x\partial^\beta_v f\Big)_{L^2}\\
	&= \Big(C_{\kappa,m,k}t^{\kappa (m+k)-1}\<v\>^l\<D_x\>^{m_0}\partial^\alpha_x\partial^\beta_vf\\
	&\qquad
	-\varepsilon t^{\kappa (m+k)}M^*M\<v\>^l\<D_x\>^{m_0}\partial^\alpha_x\partial^\beta_vf+t^{\kappa (m+k)}[M^*M,\<v\>^l\partial^\beta_v]\<D_x\>^{m_0}\partial^\alpha_xf\\
	&\qquad
	-t^{\kappa (m+k)}v\cdot\nabla_x\<v\>^l\<D_x\>^{m_0}\partial^\alpha_x\partial^\beta_vf
	+t^{\kappa (m+k)}\<v\>^l[v\cdot\nabla_x,\partial^\beta_v]\<D_x\>^{m_0}\partial^\alpha_xf\\
	&\qquad +t^{\kappa (m+k)}L\<v\>^l\partial^\alpha_x\partial^\beta_v\<D_x\>^{m_0}f -t^{\kappa (m+k)}[L,\<v\>^l\partial^\beta_v]\partial^\alpha_x\<D_x\>^{m_0}f\\
	&\qquad+t^{\kappa (m+k)}\<v\>^l\<D_x\>^{m_0}\partial^\alpha_x\partial^\beta_v\Gamma(g,f), t^{\kappa (m+k)}\<v\>^l\<D_x\>^{m_0}\partial^\alpha_x\partial^\beta_vf\Big)_{L^2},
\end{split}
\end{equation}
Notice that $\Re(v\cdot\nabla_xh,h)_{L^2}=0$ for sufficiently smooth $h$. Taking the real part and the summation on $|\alpha|\le m$, $|\beta|\le k$, we have
\begin{align}
\frac{1}{2}\partial_t\big\|t^{\kappa (m+k)}\<v\>^lf\big\|_{H^k_vH^{m+m_0}_x}^2 
+ \varepsilon\big\|Mt^{\kappa (m+k)}\<v\>^lf\big\|_{H^k_vH^{m+m_0}_x}^2
&\le C_{\kappa,m,k}(J_1+J_2+J_3+J_4+J_5),\label{eq6}
\end{align}
where
\begin{align*}
J_1 &= \big\|t^{\kappa (m+k)-\frac{1}{2}}\<v\>^lf\big\|^2_{H^k_vH^{m+m_0}_x},\\
J_2 &= \varepsilon\sum_{|\alpha|\le m,|\beta|\le k}\Re\Big(t^{\kappa (m+k)}[M^*M,\<v\>^l\partial^\beta_v]\partial^\alpha_xf,t^{\kappa (m+k)}\<v\>^l\partial^\alpha_x\partial^\beta_vf\Big)_{L^2_vH^{m_0}_x}\\
J_3 &= \sum_{|\alpha|\le m,|\beta|\le k}\Re\Big(t^{\kappa (m+k)}\<v\>^l[v\cdot\nabla_x,\partial^\beta_v]\partial^\alpha_xf, t^{\kappa (m+k)}\<v\>^l\partial^\alpha_x\partial^\beta_vf\Big)_{L^2_vH^{m_0}_x}\\
J_4 &= \sum_{|\alpha|\le m,|\beta|\le k}\Re\Big(t^{\kappa (m+k)}L\<v\>^l\partial^\alpha_x\partial^\beta_vf+t^{\kappa (m+k)}[L,\<v\>^l\partial^\beta_v]\partial^\alpha_xf, t^{\kappa (m+k)}\<v\>^l\partial^\alpha_x\partial^\beta_vf\Big)_{L^2_vH^{m_0}_x}\\
J_5 &= \sum_{|\alpha|\le m,|\beta|\le k}\Re\Big(t^{\kappa (m+k)}\<v\>^l\partial^\alpha_x\partial^\beta_v\Gamma(g,f), t^{\kappa (m+k)}\<v\>^l\partial^\alpha_x\partial^\beta_vf\Big)_{L^2_vH^{m_0}_x}.
\end{align*}
In the case of $k=m=0$, the only different term in \eqref{eq4} is the first term, in whose case $(t^{\kappa (m+k)})_t\<v\>^l\<D_x\>^{m_0}f = 0$, and hence we set \begin{align}\label{06}
J_1=0
\end{align} when $k=m=0$.
If $k+m>0$, for $\delta>0$ and sufficiently large $\kappa>>1$, by Young's inequality and \eqref{eq09},
\begin{align*}
t^{\kappa (m+k)-\frac{1}{2}} &\lesssim \big((\delta a^{1/2})^{\frac{\kappa (m+k)-\frac{1}{2}}{\kappa (m+k)}} t^{\kappa (m+k)-\frac{1}{2}}\big)^{\frac{\kappa (m+k)}{\kappa (m+k)-\frac{1}{2}}} + \big((\delta a^{1/2})^{-\frac{\kappa (m+k)-\frac{1}{2}}{\kappa (m+k)}}\big)^{2\kappa (m+k)},\\
&\lesssim \delta a^{1/2}t^{\kappa (m+k)} + C_\delta\<v\>^{-l}\<\eta\>^{-k},\\
t^{\kappa (m+k)-\frac{1}{2}} &\lesssim \big((\delta b^{1/2})^{\frac{\kappa (m+k)-\frac{1}{2}}{\kappa (m+k)}} t^{\kappa (m+k)-\frac{1}{2}}\big)^{\frac{\kappa (m+k)}{\kappa (m+k)-\frac{1}{2}}} + \big((\delta b^{1/2})^{-\frac{\kappa (m+k)-\frac{1}{2}}{\kappa (m+k)}}\big)^{2\kappa (m+k)},\\
&\lesssim \delta b^{1/2}t^{\kappa (m+k)} + C_\delta\<v\>^{-l}\<y\>^{-m}.
\end{align*} Thus $t^{\kappa (m+k)-\frac{1}{2}}\in S(\delta a^{1/2}t^{\kappa (m+k)} + C_\delta\<v\>^{-l}\<\eta\>^{-k})\cap S(\delta b^{1/2}t^{\kappa (m+k)} + C_\delta\<v\>^{-l}\<y\>^{-m})$ uniformly in $\delta, t$. 
To apply the Corollary \ref{pseudobound}, we need to check that $\delta (a^{1/2})^wt^{\kappa (m+k)}$, $C_\delta\<v\>^{-l}\<D_v\>^{-k}$ and $\delta (b^{1/2})^wt^{\kappa (m+k)}$, $C_\delta\<v\>^{-l}\<D_x\>^{-m}$ are invertible pseudodifferential operator and their seminorms are uniform in $t$ and $\delta$. The symbol of $b^w$ only depends on $v$ and $y$ and is invertible since it is multiplication on $v$ and is multiplier on $x$. $(a^{1/2})^w$ is invertible by our choice of $K_0$ in \eqref{Defofa}. Also $\<v\>^{-l}\<D_v\>^{-k}$ and $\<v\>^{-l}\<D_x\>^{-m}$ are obviously invertible and their inverse are $\<D_v\>^k\<v\>^l$ and $\<D_x\>^m\<v\>^l$ respectively. By Corollary \ref{pseudobound}, we have 
\begin{equation}\label{07}\begin{split}
	J_1 &\le \min\big\{\delta\big\|(a^{1/2})^wt^{\kappa(m+k)}\<v\>^lf\big\|^2_{H^k_vH^{m+m_0}_x} + C_\delta\big\|f\big\|^2_{L^2_vH^{m+m_0}_x},\\&\qquad\qquad \delta\big\|(b^{1/2})^wt^{\kappa(m+k)}\<v\>^lf\big\|^2_{H^k_vH^{m+m_0}_x} + C_\delta\big\|f\big\|^2_{H^k_vH^{m_0}_x}\big\}.
\end{split}
\end{equation}

For $J_2$, if $l=k=0$, then 
\begin{align}
	J_2=0. \label{08}
\end{align}If $l+k>0$, noticing that $M=\<v\>^N\<D_v\>^N\<D_x\>^N$ is invertible and $M^*\in S(\<v\>^N\<\eta\>^N\<y\>^N)$, by the estimate of commutator in appendix, we have 
\begin{align*}
[M^*M,\<v\>^l\partial^\beta_v]\in Op(\<v\>^{2N+l-1}\<\eta\>^{2N+k-1}\<y\>^{2N}).
\end{align*}
Thus Lemma \ref{pseudo1} gives that 
\begin{align}\notag
J_2 &\le \varepsilon\sum_{|\alpha|\le m,|\beta|\le k}\big|\big(t^{\kappa (m+k)}\underbrace{\<v\>^{\frac{1}{2}}\<D_v\>^{\frac{1}{2}}(M^*)^{-1}[M^*M,\<v\>^l\partial^\beta_v]\partial^\alpha_x}_{\in Op(\<v\>^{N+l-1/2}\<\eta\>^{N+k-1/2}\<y\>^{N+m})}f,\\
\notag&\qquad\qquad\qquad\qquad t^{\kappa (m+k)}\underbrace{\<v\>^{-\frac{1}{2}}\<D_v\>^{-\frac{1}{2}}M\<v\>^l\partial^\alpha_x\partial^\beta_v}_{\in Op(\<v\>^{N+l-1/2}\<\eta\>^{N+k-1/2}\<y\>^{N+m})}f\big)_{L^2_vH^{m_0}_x}\big|\\
\notag&\le \varepsilon\big\|t^{\kappa (m+k)}\big(\<v\>^{l-1/2}\<\eta\>^{k-1/2}\big)^wMf\big\|^2_{L^2_vH^{m+m_0}_x}\\
&\le \varepsilon\big(\delta\|t^{\kappa (m+k)}\<v\>^{l}\<D_v\>^kMf\big\|^2_{L^2_vH^{m+m_0}_x}
+C_\delta\|t^{\kappa m}Mf\big\|^2_{L^2_vH^{m+m_0}_x}\big),\label{09}
\end{align}where the last inequality follows from  Corollary \ref{pseudobound}, $\<v\>^{l-1/2}\<\eta\>^{k-1/2}\in S(\delta\<v\>^{l}\<\eta\>^{k}+C_\delta)$ and $t\le 1$.

For $J_3$, if $k=0$, then \begin{align}
	J_3=0.\label{010}
\end{align} If $k\neq 0$, noticing that \begin{align*}
[v\cdot\nabla_x,\partial^\beta_v]\in Op(\<\eta\>^{k-1}\<y\>),
\end{align*}
and
\begin{align*}
\|\<D_v\>^{k-1}\<D_x\>h\|_{L^2} &= \|\<\eta\>^{k-1}\<y\>h\|_{L^2}\\
&\le \delta\|\<D_v\>^kh\|_{L^2} + C_\delta\|\<D_x\>^kh\|_{L^2},
\end{align*}then we know
\begin{align}\label{011}
\notag J_3 &\le \sum_{|\alpha|\le m,|\beta|\le k}\big\|t^{\kappa (m+k)}\underbrace{\<v\>^l[v\cdot\nabla_x,\partial^\beta_v]}_{\in Op (\<\eta\>^{k-1}\<y\>\<v\>^l)}\partial^\alpha_xf\big\|_{L^2_vH^{m_0}_x}\big\|t^{\kappa (m+k)}\<v\>^l\partial^\alpha_x\partial^\beta_vf\big\|_{L^2_vH^{m_0}_x}\\
\notag&\le \big\|t^{\kappa (m+k)}\<D_v\>^{k-1}\<D_x\>\<v\>^lf\big\|_{L^2_vH^{m+m_0}_x}\big\|t^{\kappa (m+k)}\<D_v\>^k\<v\>^lf\big\|_{L^2_vH^{m+m_0}_x}\\
&\le \delta\big\|t^{\kappa (m+k)}\<D_v\>^k\<v\>^lf\big\|^2_{L^2_vH^{m+m_0}_x}+C_\delta\big\|t^{\kappa (m+k)}\<D_x\>^k\<v\>^lf\big\|^2_{L^2_vH^{m+m_0}_x}.
\end{align}

For $J_4$, when $l=0$, by \eqref{equivalent_norm} and \eqref{estimateofL}, there exists $\nu_0>0$ such that 
\begin{align}
	&\quad\,\sum_{|\alpha|\le m,|\beta|\le k}\Re\Big(t^{\kappa (m+k)}L\<v\>^l\partial^\alpha_x\partial^\beta_v\<D_x\>^{m_0}f, t^{\kappa (m+k)}\<v\>^l\partial^\alpha_x\partial^\beta_v\<D_x\>^{m_0}f\Big)_{L^2}\notag\\
	&\le 
	-\nu_0\sum_{|\alpha|\le m,|\beta|\le k}\big(\big\|t^{\kappa (m+k)}\partial^\alpha_x\partial^\beta_v\<D_x\>^{m_0}f\big\|_{N^{s,\gamma}}^2+C\big\|t^{\kappa (m+k)}\underbrace{\partial^\alpha_x\partial^\beta_v}_{\in Op(\<y\>^m\<\eta\>^k)}\<D_x\>^{m_0}f\big\|^2_{L^2}\big)\notag\\
	&\le 
	-\nu_0\sum_{|\alpha|\le m,|\beta|\le k}\big\|t^{\kappa (m+k)}(a^{1/2})^w\partial^\alpha_x\partial^\beta_v\<D_x\>^{m_0}f\big\|_{L^2}^2+C\big\|t^{\kappa (m+k)}\<D_v\>^k\<D_x\>^{m+m_0}f\big\|^2_{L^2}.\label{eq35}
\end{align}The last inequality follows from Corollary \ref{pseudobound}.
For the first term in \eqref{eq35}, we observe 
\begin{align}
	&\quad\,\sum_{|\alpha|\le m,|\beta|\le k}\big\|t^{\kappa (m+k)}(a^{1/2})^w\partial^\alpha_x\partial^\beta_vf\big\|^2_{L^2_vH^{m_0}_x}\notag\\
&\ge \frac{1}{2}\sum_{|\alpha|\le m,|\beta|\le k}\big\|t^{\kappa (m+k)}\partial^\beta_v(a^{1/2})^w\partial^\alpha_xf\big\|^2_{L^2_vH^{m_0}_x}	- \sum_{|\alpha|\le m,|\beta|\le k}\big\|t^{\kappa (m+k)}[(a^{1/2})^w,\partial^\beta_v]\partial^\alpha_xf\big\|^2_{L^2_vH^{m_0}_x}\notag\\
&\ge \frac{1}{2C}\big\|t^{\kappa (m+k)}\<D_v\>^k\<D_x\>^m(a^{1/2})^wf\big\|^2_{L^2_vH^{m_0}_x}	- \sum_{|\alpha|\le m}\big\|t^{\kappa (m+k)}\<D_v\>^{k-1}\<D_x\>^m(a^{1/2})^w\partial^\alpha_xf\big\|^2_{L^2_vH^{m_0}_x}\notag\\
&\ge \frac{1}{4C}\big\|t^{\kappa (m+k)}\<D_v\>^k\<D_x\>^m(a^{1/2})^wf\big\|^2_{L^2_vH^{m_0}_x} - C\big\|t^{\kappa m}\<D_x\>^m(a^{1/2})^wf\big\|^2_{L^2_vH^{m_0}_x},\label{eq36}
\end{align}since $\<\eta\>^{k-1}\in S(\delta\<\eta\>^k+C_\delta)$.

For the second term in $J_4$, we observe that from the definition of $L$, when $l=0$, 
\begin{align*}
	[L,\partial^\beta_v]\partial^\alpha_xf
	&= -\sum_{\beta_1+\beta_2+\beta_3=\beta,\beta_1<\beta}C_{\beta_1,\beta_2,\beta_3,\beta}\big(\T(\partial^{\beta_2}_v\mu^{1/2},\partial^{\beta_1}_v\partial^\alpha_xf,\partial^{\beta_3}_v\mu^{1/2})+\T(\partial^{\beta_1}_v\partial^\alpha_xf,\partial^{\beta_2}_v\mu^{1/2},\partial^{\beta_3}_v\mu^{1/2})\big)
\end{align*}
Thus by \eqref{estiofGamma}, when $l=0$, 
\begin{equation}\begin{split}
		&\quad\,\sum_{|\alpha|\le m,|\beta|\le k}\Re\Big(t^{\kappa (m+k)}[L,\<v\>^l\partial^\beta_v]\partial^\alpha_xf, t^{\kappa (m+k)}\<v\>^l\partial^\alpha_x\partial^\beta_vf\Big)_{L^2_vH^{m_0}_x}\\
		&\le C\sum_{|\alpha|\le m,|\beta|\le k}\sum_{|\beta_1|\le|\beta|-1}\big\|t^{\kappa (m+k)}(a^{1/2})^w\partial^{\beta_1}_v\partial^\alpha_x\<D_x\>^{m_0}f\big\|_{L^2} \big\|t^{\kappa (m+k)}(a^{1/2})^w\partial^\alpha_x\partial^\beta_v\<D_x\>^{m_0}f\big\|_{L^2}\\
		&\le C_\delta\big\|t^{\kappa (m+k)}\<D_v\>^{k-1}(a^{1/2})^wf\big\|_{L^2_vH^{m+m_0}_x}^2 +\delta\sum_{|\alpha|\le m,|\beta|\le k}\big\|t^{\kappa (m+k)}(a^{1/2})^w\partial^\alpha_x\partial^\beta_vf\big\|^2_{L^2_vH^{m_0}_x}\label{013}\\
		&\le C_\delta\big\|t^{\kappa (m+k)}(a^{1/2})^wf\big\|_{L^2_vH^{m+m_0}_x}^2 +2\delta\big\|t^{\kappa (m+k)}\<D_v\>^k(a^{1/2})^wf\big\|^2_{L^2_vH^{m+m_0}_x}. 
\end{split} 
\end{equation}
Therefore, if $l=0$, by substituting \eqref{eq35}\eqref{eq36}\eqref{013} into $J_4$, choosing $\delta$ small, we have 
\begin{equation}\label{eq37}
	J_4 \le -\frac{\nu_0}{4C}\big\|t^{\kappa (m+k)}\<D_v\>^k(a^{1/2})^wf\big\|^2_{L^2_vH^{m+m_0}_x} + C\big\|t^{\kappa m}(a^{1/2})^wf\big\|^2_{L^2_vH^{m+m_0}_x}
	+C\big\|t^{\kappa (m+k)}\<D_v\>^kf\big\|^2_{L^2_vH^{m+m_0}_x}.
\end{equation}
Notice that the last term is bounded above by $J_1$.

If $k=0$, then by \eqref{estimateofL}\eqref{widehat}, 
\begin{align}
J_4 &= \sum_{|\alpha|\le m}\Re\Big(t^{\kappa m}\<v\>^lL\partial^\alpha_x\<D_x\>^{m_0}f, t^{\kappa m}\<v\>^l\partial^\alpha_x\<D_x\>^{m_0}f\Big)_{L^2}\notag\\
&\le -\nu_0\sum_{|\alpha|\le m}\|t^{\kappa m}\partial^\alpha_x\<D_x\>^{m_0}f\|^2_{N^{s,\gamma}_l} - C\sum_{|\alpha|\le m}\|t^{\kappa m}\partial^\alpha_x\<D_x\>^{m_0}f\|^2_{L^2}\notag\\
&\le  -\nu_0 \|t^{\kappa m}\<D_x\>^{m+m_0}f\|^2_{N^{s,\gamma}_l}+\|t^{\kappa m}f\|^2_{L^2_vH^{m+m_0}_x}. \label{014}
\end{align}Here the last term is bounded above by $J_1$.

For $J_5$, we compute 
\begin{align*}
\partial^\alpha_x\partial^\beta_v\Gamma(g,f) &= \sum_{\alpha_1+\alpha_2=\alpha}\sum_{\beta_1+\beta_2+\beta_3=\beta}C_{\alpha_1,\alpha_2,\beta_1,\beta_2,\beta_3}\T(\partial^{\alpha_1}_x\partial^{\beta_2}_vg,\partial^{\alpha_2}_x\partial^{\beta_2}_vf,\partial^{\beta_3}_v(\mu^{1/2})).
\end{align*}
Then by the estimate \eqref{estiofGamma} of $\Gamma$,
\begin{align*}
J_5 
&\le C\sum_{|\alpha|\le m,|\beta|\le k} \sum_{\alpha_1+\alpha_2=\alpha}\sum_{\beta_1+\beta_2+\beta_3=\beta}
\|t^{\kappa(|\alpha_1|+|\beta_1|)}\<v\>^l\partial^{\alpha_1}_x\partial^{\beta_1}_vg\|_{L^2_vH^{m_0}_x}\|t^{\kappa(|\alpha_2|+|\beta_2|)}\partial^{\alpha_2}_x\partial^{\beta_2}_v\<D_x\>^{m_0}f\|_{N^{s,\gamma}_l}\\
&\qquad\qquad\qquad\qquad\qquad\qquad\qquad\qquad\qquad\qquad\qquad\qquad\times\|t^{\kappa (m+k)}\partial^\alpha_x\partial^\beta_v\<D_x\>^{m_0}f\|_{N^{s,\gamma}_l},
\end{align*}
Here we will apply \eqref{assumong} on $g$ and divide the summation into two parts: $\alpha_1=\beta_1=0$ and else. In the second case, we have $|\alpha_2|<m$ or $|\beta_2|<k$ and then one can apply Young's inequality to second term to eliminate derivatives on $x$ or $v$. When $m=l=0$, by \eqref{equivalent_norm} and Corollary \ref{pseudobound}, 
\begin{align}
J_5&\le C\sum_{|\beta|\le k}\Big(\varepsilon_0\sum_{|\beta_2|\le k}\|t^{\kappa k}(a^{1/2})^w\partial^{\beta_2}_vf\|_{L^2_vH^{m_0}_x}+\sum_{|\beta_2|\le k-1}\varepsilon_1\|t^{\kappa k}(a^{1/2})^w\partial^{\beta_2}_vf\|_{L^2_vH^{m_0}_x}\Big)\|t^{\kappa k}(a^{1/2})^w\partial^\beta_vf\|_{L^2_vH^{m_0}_x}\notag\\
&\le C\Big(\varepsilon_0\|t^{\kappa k}\<D_v\>^k(a^{1/2})^wf\|_{L^2_vH^{m_0}_x}+\varepsilon_1\|t^{\kappa (k-1)}\<D_v\>^{k-1}(a^{1/2})^wf\|_{L^2_vH^{m_0}_x}\Big)\|t^{\kappa k}\<D_v\>^k(a^{1/2})^wf\|_{L^2_vH^{m_0}_x}\notag\\
&\le (\varepsilon_0C+\varepsilon_1C\delta)\|t^{\kappa k}\<D_v\>^k(a^{1/2})^wf\|^2_{L^2_vH^{m_0}_x} + \varepsilon_1C_\delta\|(a^{1/2})^wf\|^2_{L^2_vH^{m_0}_x}.\label{016}
\end{align}
In the case of $k=0$, we use \eqref{widehat} and Young's inequality to get 
\begin{align}
J_5 &\le C\sum_{|\alpha|\le m} \sum_{\alpha_1+\alpha_2=\alpha}
\|t^{\kappa|\alpha_1|}\<v\>^l\partial^{\alpha_1}_xg\|_{L^2_vH^{m_0}_x}\|t^{\kappa|\alpha_2|}\partial^{\alpha_2}_x\<D_x\>^{m_0}f\|_{N^{s,\gamma}_l}\|t^{\kappa m}\partial^\alpha_x\<D_x\>^{m_0}f\|_{N^{s,\gamma}_l}\notag\\
&\le (\varepsilon_0C+\varepsilon_1C\delta)\sum_{|\alpha|\le m}\|t^{\kappa m}\partial^\alpha_x\<D_x\>^{m_0}f\|_{N^{s,\gamma}_l}^2+\varepsilon_1C_\delta\sum_{|\alpha|\le m-1}\|t^{\kappa m}\partial^\alpha_x\<D_x\>^{m_0}f\|^2_{N^{s,\gamma}_l}\notag\\
&\le 
(\varepsilon_0C+2\varepsilon_1C\delta)\|t^{\kappa m}\<D_x\>^{m+m_0}f\|_{N^{s,\gamma}_l}^2+\varepsilon_1C_\delta\|t^{\kappa m}\<D_x\>^{m_0}f\|_{N^{s,\gamma}_l}^2,\label{017}
\end{align}
Similarly, if $m=k=0$, we have 
\begin{align}
	\label{018}J_5\le \varepsilon_0C\|\<D_x\>^{m_0}f\|_{N^{s,\gamma}_l}^2. 
\end{align}

In a summary, if $m=l=0$, substituting \eqref{07}\eqref{09}\eqref{011}\eqref{eq37}\eqref{016} into \eqref{eq6},
\begin{align*}
	&\quad\,\partial_t\|t^{\kappa k}f\|_{H^k_vH^{m_0}_x}^2 
	+ \varepsilon\|Mt^{\kappa k}f\|_{H^k_vH^{m_0}_x}^2
	+\nu_0\big\|t^{\kappa (m+k)}\<D_v\>^k(a^{1/2})^wf\big\|_{L^2_vH^{m_0}_x}^2\\
	&\le C_{\kappa,k}\big(\delta\|(a^{1/2})^wt^{\kappa k}f\|^2_{H^k_vH^{m_0}_x} + C_\delta\|f\|^2_{L^2_vH^{m_0}_x}\\
	&\qquad\quad\quad+\varepsilon\delta\|t^{\kappa k}\<D_v\>^kMf\|^2_{L^2_vH^{m_0}_x}
	+\varepsilon C_\delta\|Mf\|^2_{L^2_vH^{m_0}_x}\\
	&\qquad\quad\quad+\delta\|t^{\kappa k}\<D_v\>^{k}f\|^2_{L^2_vH^{m_0}_x} + C_\delta\|t^{\kappa k}\<D_x\>^kf\|^2_{L^2_vH^{m_0}_x}\\
	&\qquad\quad\quad + \|(a^{1/2})^wf\|^2_{L^2_vH^{m_0}_x}
	+J_1\\
	&\qquad\quad\quad+(\varepsilon_0C+\varepsilon_1C\delta)\|t^{\kappa k}(a^{1/2})^wf\|^2_{H^k_vH^{m_0}_x} + \varepsilon_1C_\delta\|(a^{1/2})^wf\|^2_{L^2_vH^{m_0}_x}\big).
\end{align*}
Choosing constant $\delta$ and $\varepsilon_0$ sufficiently small, applying $\varepsilon_1\le 1$, noting $\|\cdot\|_{L^2}\lesssim \|(a^{1/2})^w(\cdot)\|_{L^2}$, we obtain 
\begin{align*}
	\notag&\quad\,\partial_t\|t^{\kappa k}f\|_{H^k_vH^{m_0}_x}^2 
	+ \varepsilon\|Mt^{\kappa k}f\|_{H^k_vH^{m_0}_x}^2
	+\nu_0\|(a^{1/2})^wt^{\kappa k}f\|_{H^k_vH^{m_0}_x}^2\\
	&\le C_{\kappa,k}\big(\varepsilon\|Mf\|^2_{L^2_vH^{m_0}_x} + \|(a^{1/2})^wf\|^2_{L^2_vH^{m_0}_x} + \|t^{\kappa k}\<D_x\>^kf\|^2_{L^2_vH^{m_0}_x}\big).
\end{align*}
If $k=l=0$, a similar computation by substituting \eqref{07}\eqref{08}\eqref{010}\eqref{014}\eqref{017} into \eqref{eq6}, we obtain 
\begin{align*}
	\notag&\quad\,\partial_t\|t^{\kappa m}f\|_{L^2_vH^{m+m_0}_x}^2 
	+ \varepsilon\|Mt^{\kappa m}f\|_{L^2_vH^{m+m_0}_x}^2
	+\nu_0\|(a^{1/2})^wt^{\kappa m}f\|_{L^2_vH^{m+m_0}_x}^2\\
	&\le \delta\|(b^{1/2})^wt^{\kappa m}f\|^2_{L^2_vH^{m+m_0}_x} + C_{\delta}\|(a^{1/2})^wf\|^2_{L^2_vH^{m_0}_x}.
\end{align*}
The last case is the simplest one: $m=k=0$, we substitute \eqref{06}\eqref{09}\eqref{010}\eqref{014}\eqref{018} into \eqref{eq6}, then 
\begin{align}\label{23}
	\quad\,\partial_t\|\<v\>^lf\|_{L^2_vH^{m_0}_x}^2 
	+ \varepsilon\|M\<v\>^lf\|_{L^2_vH^{m_0}_x}^2
	+\nu_0\|\<D_x\>^{m_0}f\|_{N^{s,\gamma}_l}^2
	&\le  C\big(\varepsilon\|Mf\|^2_{L^2_vH^{m_0}_x}+\|f\|^2_{L^2_vH^{m_0}_x}\big).
\end{align}
Integrate \eqref{23} on $t$ and apply \eqref{eq0}, we have 
\begin{align*}
	\sup_{0\le t\le T}\|\<v\>^lf\|_{L^2_vH^{m_0}_x}^2 
	+ \varepsilon\int^T_0\|M\<v\>^lf\|_{L^2_vH^{m_0}_x}^2\,dt
	+\nu_0\int^T_0\|\<D_x\>^{m_0}f\|_{N^{s,\gamma}_l}^2\,dt
	&\le C\varepsilon_0^2. 
\end{align*} 

\end{proof}

Next we establish the smoothing estimate of $x$. The idea here is to consider the Poisson bracket between $v\cdot\nabla_x$ and our chosen function $\theta$. It will give us $b^{1/2}$ regularity for solution $f$ to \eqref{lineareqvarepsilon}. 
\begin{Thm}\label{Thm33}
Let $T\in(0,1]$, $m_0>\frac{d}{2}$. Assume $g$ satisfies \eqref{assumong}.		Let $f$ be the solution to equation \eqref{lineareqvarepsilon}.
	Let $m\in\N$ $(m\ge 1)$ to be the index for derivative of $x$ and weight $\<v\>$ respectively. For sufficiently large $\kappa>1$,
	\begin{align}\notag
		&\quad\,C_0\sup_{0\le t\le T}\|t^{\kappa m}f(t)\|_{L^2_vH^{m+m_0}_x}^2 
		+ C_0\varepsilon\int^T_0\|Mt^{\kappa m}f\|_{L^2_vH^{m+m_0}_x}^2\,dt\\\label{30}&\quad
		+C_0\nu_0\int^T_0\|(a^{1/2})^wt^{\kappa m}f\|_{L^2_vH^{m+m_0}_x}^2\,dt+\int^T_0\|t^{\kappa m}(b^{1/2})^wf\|^2_{L^2_vH^{m+m_0}_x}\,dt\\
		&\le C_0C\varepsilon_0^2.\notag
	\end{align}
\end{Thm}
\begin{proof}
Assume $\gamma+2s>0$. 
We define constants $\delta_1,\delta_2>0$ as the followings. 
Let $c = \max\{0,-\frac{\gamma}{2}\}$. If $\gamma+3s-c\le 1$, we let 
\begin{align*}
\delta_1 = \frac{-s+c}{\gamma+s+c-2},\qquad\delta_2 = \frac{-1}{\gamma+s+c-2}.
\end{align*}
If $\gamma+3s-c\ge 1$, we let
\begin{align*}
\delta_1 = \frac{s-c}{2s-2c+1},\qquad\delta_2 = \frac{1}{2s-2c+1}.
\end{align*}
Then in each case, since $s-c\in(0,1)$, by direct calculation we have the following estimates. 
\begin{align}
\delta_1\le 1,\quad\delta_1\le \delta_2,\quad\delta_1-\frac{1}{2}+\frac{\delta_2}{2}\le 0,\quad   \frac{\delta_1}{\delta_2}+c\le s,\quad \frac{\delta_1-1}{\delta_2}\le\gamma+2s-2.
\end{align} 
Let $\chi_0$ be a smooth cutoff function such that $\chi_0(z)$ equal to $1$ when $|z|<\frac{1}{2}$ and equal to $0$ when $|z|\ge 1$. Define 
\begin{align*}
	b(v,y) &= (1+|v|^2+|y|^2+|v\wedge y|^2)^{\delta_1},\\
\chi(v,\eta) &= \chi_0\bigg(\frac{1+|v|^2+|\eta|^2+|v\wedge\eta|^2}{(1+|v|^2+|y|^2+|v\wedge y|^2)^{\delta_2}}\bigg),
\end{align*}and
\begin{align*}
\theta(v,\eta) = (1+|v|^2+|y|^2+|v\wedge y|^2)^{\delta_1-1} (y\cdot\eta+(v\wedge y)\cdot(v\wedge\eta))\chi(v,\eta).
\end{align*}
We first check that $\theta\in S(1)$ and compute $\{\theta,v\cdot y\}$. 
Indeed, using the support of $\chi$, we have 
\begin{align*}
|\theta(v,\eta)|&\lesssim (1+|v|^2+|y|^2+|v\wedge y|^2)^{\delta_1-1/2}(|\eta|+|v\wedge\eta|)\chi(v,\eta)\\
&\lesssim (1+|v|^2+|y|^2+|v\wedge y|^2)^{\delta_1-1/2+\delta_2/2}\\& \lesssim 1,
\end{align*}
since $\delta_1-1/2+\delta/2\le 0$. Same argument are valid for the derivatives of $\theta$ by Leibniz's formula and hence $\theta\in S(1)$. Therefore $\theta^w$ is a linear bounded operator on $L^2$. On the other hand, 
\begin{align}
\{\theta, v\cdot y\}\notag
&= \partial_\eta\theta\cdot\partial_v(v\cdot y)\\
&= (1+|v|^2+|y|^2+|v\wedge y|^2)^{\delta_1-1}(|y|^2+|v\wedge y|^2)\chi(v,\eta)\notag\\
&\qquad + (1+|v|^2+|y|^2+|v\wedge y|^2)^{\delta_1-1}(y\cdot\eta+(v\wedge y)\cdot(v\wedge\eta))\chi_\eta\cdot y\notag\\
&= (1+|v|^2+|y|^2+|v\wedge y|^2)^{\delta_1} + I_1+I_2+I_3,\label{eqbracket}\end{align}
where 
\begin{align*}
I_1 &=  (1+|v|^2+|y|^2+|v\wedge y|^2)^{\delta_1}(\chi-1),\\
I_2 &= - \<v\>^2(1+|v|^2+|y|^2+|v\wedge y|^2)^{\delta_1-1}\chi, \\
I_3 &= (1+|v|^2+|y|^2+|v\wedge y|^2)^{\delta_1-1}(y\cdot\eta+(v\wedge y)\cdot(v\wedge\eta))\chi_\eta\cdot y. 
\end{align*}
On the support of $\chi-1$, we have $(1+|v|^2+|y|^2+|v\wedge y|^2)^{\delta_2}\le 2(1+|v|^2+|\eta|^2+|v\wedge\eta|^2)$. Since $\delta_1>0$, 
\begin{align*}
|I_1| &\lesssim (1+|v|^2+|\eta|^2+|v\wedge\eta|^2)^{\frac{\delta_1}{\delta_2}}\\
&\lesssim \<v\>^\gamma(1+|v|^2+|\eta|^2+|v\wedge\eta|^2)^{\frac{\delta_1}{\delta_2}+\max\{{0,-\frac{\gamma}{2}}\}}\\&\lesssim a(v,\eta),
\end{align*}since $\frac{\delta_1}{\delta_2}+\max\{0,-\frac{\gamma}{2}\}\le s$. Since the support of  derivative of $\chi-1$ is contained in the support of $\chi-1$, same control holds true for the derivatives of $I_1$ by Leibniz's formula, and hence $I_1\in S(a)$. 
Similarly, since $\delta_1\le1$
\begin{align*}
|I_2| 
&\le \<v\>^2(1+|v|^2+|\eta|^2+|v\wedge \eta|^2)^{(\delta_1-1)/\delta_2}\\
&\le \<v\>^{2+(\delta_1-1)/\delta_2}\\&\le a(v,\eta),
\end{align*}
since $2+(\delta_1-1)/\delta_2\le\gamma+2s$. Same control is valid for the derivatives of $I_2$ by Leibniz's formula, and hence $I_2\in S(a)$. 
For the last term, since $\delta_1-\delta_2\le0$, 
\begin{align*}
|I_3|&\le (1+|v|^2+|y|^2+|v\wedge y|^2)^{\delta_1}(|\eta|+|v\wedge\eta|)|\chi_\eta|\\
&\lesssim (1+|v|^2+|y|^2+|v\wedge y|^2)^{\delta_1-\delta_2}(|\eta|+|v\wedge\eta|)^2\1_{\text{support of $\chi$}}\\
&\lesssim (1+|v|^2+|\eta|^2+|v\wedge\eta|^2)^{\frac{\delta_1}{\delta_2}}\\
&\lesssim a(v,\eta), 
\end{align*}where the last inequality follows from $I_1$.
Therefore, $I_1+I_2+I_3\in S(a)$ and hence \eqref{compostion}\eqref{eqbracket} gives that for $h\in\S$, 
\begin{equation}\begin{split}\label{eqy}
\|(b^{1/2})^wh\|^2_{L^2}=\big(b^{1/2}\widehat{h},\widehat{h}\big)_{L^2} &= \big(\{\theta,v\cdot y\}^w\widehat{h},\widehat{h}\big)_{L^2} + \big((I_1+I_2+I_3)^w\widehat{h},\widehat{h}\big)_{L^2}\\
&\le 2\Re(2\pi iv\cdot y\widehat{h},\theta^w\widehat{h})_{L^2} + C\|(a^{1/2})^w\widehat{h}\|^2_{L^2}\\
&\le 2\Re(v\cdot\nabla_xh,(\theta^w\widehat{h})^\vee)_{L^2} + C\|(a^{1/2})^wh\|^2_{L^2},
\end{split}
\end{equation}where $\widehat{h}$ is the Fourier transform of $h$ on spatial variable $x$. By density, estimate \eqref{eqy} is valid for $h\in H(\<v\>^{1+\gamma+2s}\<\eta\>^{2s}\<y\>)$.

Now we take multi-index $|\alpha|\le m$ and $f$ to be the solution to equation \eqref{lineareqvarepsilon}. Then we can substitute $t^{\kappa m}\<D_x\>^{m_0}\partial^\alpha_xf$ into \eqref{eqy}, 
\begin{equation}\label{eqK}\begin{split}
&\quad\,\|t^{\kappa m}(b^{1/2})^w\<D_x\>^{m_0}\partial^\alpha_xf\|_{L^2_{y,v}}\\
&\le 
2\Re\big(t^{\kappa m}\<D_x\>^{m_0}\partial^\alpha_x(v\cdot\nabla_x f),(t^{\kappa m}\theta^w\<y\>^{m_0}\widehat{\partial^\alpha_xf})^{\vee}\big)_{L^2} + C\|t^{\kappa m}(a^{1/2})^wf\|^2_{L^2_vH^{m+m_0}_x}.
\end{split}  
\end{equation}
By equation \eqref{lineareqvarepsilon}, we have 
\begin{align*}
&\quad\,\Re\left(t^{\kappa m}\<D_x\>^{m_0}\partial^\alpha_x\big(v\cdot\nabla_x f\big),(t^{\kappa m}\theta^w\<y\>^{m_0}\widehat{\partial^\alpha_xf})^{\vee}\right)_{L^2}\\
&=\Re\left(t^{\kappa m}\<D_x\>^{m_0}\partial^\alpha_x\big(-\partial_tf -\varepsilon M^*Mf+Lf+\Gamma(g,f)\big),(t^{\kappa m}\theta^w\<y\>^{m_0}\widehat{\partial^\alpha_xf})^{\vee}\right)_{L^2}\\
&= K_1+K_2+K_3+K_4+K_5,
\end{align*}
where 
\begin{align*}
K_1 &= -\frac{1}{2}\partial_t\left(t^{\kappa m}\<D_x\>^{m_0}\partial^\alpha_xf,(t^{\kappa m}\theta^w\<y\>^{m_0}\widehat{\partial^\alpha_xf})^{\vee}\right)_{L^2}\\
K_2 &= -C_{\kappa,m}\Re\left(t^{\kappa m-1}\<D_x\>^{m_0}\partial^\alpha_xf,(t^{\kappa m}\theta^w\<y\>^{m_0}\widehat{\partial^\alpha_xf})^{\vee}\right)_{L^2}\\
K_3 &= -\varepsilon\Re\left(t^{\kappa m}\<D_x\>^{m_0}\partial^\alpha_x M^*Mf,(t^{\kappa m}\theta^w\<y\>^{m_0}\widehat{\partial^\alpha_xf})^{\vee}\right)_{L^2}\\
K_4 &= \Re\left(t^{\kappa m}\<D_x\>^{m_0}\partial^\alpha_xLf,(t^{\kappa m}\theta^w\<y\>^{m_0}\widehat{\partial^\alpha_xf})^{\vee}\right)_{L^2}\\
K_5 &= \Re\left(t^{\kappa m}\<D_x\>^{m_0}\partial^\alpha_x\Gamma(g,f),(t^{\kappa m}\theta^w\<y\>^{m_0}\widehat{\partial^\alpha_xf})^{\vee}\right)_{L^2}.
\end{align*}
For $K_2$, for $\delta>0$, we choose $\kappa$ so large that 
\begin{align*}
t^{\kappa m-1} &\lesssim \big((\delta b^{1/2})^{\frac{\kappa m-1}{\kappa m}}t^{\kappa m-1}\big)^{\frac{\kappa m}{\kappa m-1}} + \big((\delta b^{1/2})^{-\frac{\kappa m-1}{\kappa m}}\big)^{\kappa m},\\
&\lesssim \delta t^{\kappa m}b^{1/2} + C_\delta\<v\>^{-l}\<y\>^{-m},
\end{align*}by \eqref{eq09}. Then we have $t^{\kappa m-1}\in S(\delta t^{\kappa m}b^{1/2} + C_\delta\<v\>^{-l}\<y\>^{-m})$ uniformly in $t,\delta$. Hence by lemma \ref{pseudobound} and $\theta\in S(1)$, $|\alpha|\le m$,
\begin{align*}
K_2 &\le C_{\kappa,m}\big\|t^{\kappa m-1}\<D_x\>^{m_0}\partial^\alpha_xf\big\|_{L^2}\big\|(t^{\kappa m}\theta^w\<y\>^{m_0}\widehat{\partial^\alpha_xf})^{\vee}\big\|_{L^2}\\
&\le \delta \|t^{\kappa m}(b^{1/2})^wf\big\|^2_{L^2_vH^{m+m_0}_x}+C_{\delta}\|f\|^2_{L^2_vH^{m_0}_x}+C_\delta\big\|t^{\kappa m}f\big\|^2_{L^2_vH^{m+m_0}_x}.
\end{align*}For $K_3$ and $K_4$, we use $L\in S(a)$ and $M=\<v\>^N\<D_v\>^N\<D_x\>^N$ to obtain 
\begin{align*}
|K_3| + |K_4| \le \varepsilon C\|t^{\kappa m}Mf\|^2_{L^2_vH^{m+m_0}_x} + C\|t^{\kappa m}(a^{1/2})^wf\|^2_{L^2_vH^{m+m_0}_x}. 
\end{align*}
For $K_5$, we compute 
\begin{align*}
\partial^\alpha_x\Gamma(g,f) &= \sum_{\beta\le\alpha}\begin{pmatrix}
\alpha\\\beta
\end{pmatrix}\Gamma(\partial^\beta_xg,\partial^{\alpha-\beta}_xf).
\end{align*}
Then by the estimate \eqref{estiofGamma} of $\Gamma$ and assumption \eqref{assumong} on $g$, a similar computation to \eqref{016} with the help of \eqref{widehat} yields that 
\begin{align*}
K_5 &\le C\Big|\sum_{\beta\le\alpha}\Big(t^{\kappa m}\<D_x\>^{m_0}\Gamma(\partial^\beta_xg,\partial^{\alpha-\beta}_xf),(t^{\kappa m}\theta^w\<y\>^{m_0}\widehat{\partial^\alpha_xf})^{\vee}\Big)_{L^2}\Big|\\
&\le C\sum_{\beta\le\alpha}\|t^{\kappa |\beta|}\partial^\beta_xg\|_{L^2_vH^{m_0}_x}\|t^{\kappa|\alpha-\beta|}\<D_x\>^{m_0}\partial^{\alpha-\beta}_xf\|_{N^{s,\gamma}}\|t^{\kappa m}\<D_x\>^{m+m_0}f\|_{N^{s,\gamma}}\\
&\le (\varepsilon_0C+\varepsilon_1C\delta)\|t^{\kappa m}(a^{1/2})^w\<D_x\>^{m+m_0}f\|^2_{L^2} +\varepsilon_1C_\delta\|(a^{1/2})^w\<D_x\>^{m_0}f\|^2_{L^2} .
\end{align*}
Substitute these estimate into \eqref{eqK} and choose $\delta$ sufficiently small, we see  
\begin{align*}
&\quad\,\frac{1}{2}\|t^{\kappa m}(b^{1/2})^w\<D_x\>^{m_0}\partial^\alpha_xf\|^2_{L^2}\\
&\le 
-C\partial_t\left(t^{\kappa m}\<D_x\>^{m_0}\partial^\alpha_xf,(t^{\kappa m}\theta^w\<y\>^{m_0}\widehat{\partial^\alpha_xf})^{\vee}\right)_{L^2}
 + C\|t^{\kappa m}(a^{1/2})^wf\|^2_{L^2_vH^{m+m_0}_x}\\
&\quad+C\|f\|^2_{L^2_vH^{m_0}_x}
+\varepsilon C\|t^{\kappa m}Mf\|^2_{L^2_vH^{m+m_0}_x}+\varepsilon_1C\|(a^{1/2})^wf\|^2_{L^2_vH^{m_0}_x}.
\end{align*}
Taking summation on $|\alpha|\le m$, 
\begin{align}
&\quad\,\|t^{\kappa m}(b^{1/2})^wf\|^2_{L^2_vH^{m+m_0}_x}\notag\\
&\le 
-C\sum_{|\alpha|\le m}\partial_t\left(t^{\kappa m}\<D_x\>^{m_0}\partial^\alpha_xf,(t^{\kappa m}\theta^w\<y\>^{m_0}\widehat{\partial^\alpha_xf})^{\vee}\right)_{L^2}
+ C\|t^{\kappa m}(a^{1/2})^wf\|^2_{L^2_vH^{m+m_0}_x}\label{eq15}\\
&\quad+C\|f\|^2_{L^2_vH^{m_0}_x}
+\varepsilon C\|t^{\kappa m}Mf\|^2_{L^2_vH^{m+m_0}_x}+\varepsilon_1C\|(a^{1/2})^wf\|^2_{L^2_vH^{m_0}_x}.\notag
\end{align}

We multiply \eqref{21} with a large constant $C_0$ and add to \eqref{eq15}, then  
\begin{align*}
&\quad\,C_0\partial_t\|t^{\kappa m}f\|_{L^2_vH^{m+m_0}_x}^2 
+ C_0\varepsilon\|Mt^{\kappa m}f\|_{L^2_vH^{m+m_0}_x}^2
\\
&\quad\,+C_0\nu_0\|(a^{1/2})^wt^{\kappa m}f\|_{L^2_vH^{m+m_0}_x}^2+\|(b^{1/2})^wt^{\kappa m}f\|^2_{L^2_vH^{m+m_0}_x}\\
&\le C_0\delta\|(b^{1/2})^wt^{\kappa m}f\|^2_{L^2_vH^{m+m_0}_x} + C_0C_{\delta}\|(a^{1/2})^wf\|^2_{L^2_vH^{m_0}_x}\\
&\quad-C\sum_{|\alpha|\le m}\partial_t\left(t^{\kappa m}\<D_x\>^{m_0}\partial^\alpha_xf,(t^{\kappa m}\theta^w\<y\>^{m_0}\widehat{\partial^\alpha_xf})^{\vee}\right)_{L^2}
+ C\|(a^{1/2})^wt^{\kappa m}f\|^2_{L^2_vH^{m+m_0}_x}\\
&\quad+C\|f\|^2_{L^2_vH^{m_0}_x}
+\varepsilon C\|t^{\kappa m}Mf\|^2_{L^2_vH^{m+m_0}_x}+\varepsilon_1C\|(a^{1/2})^wf\|^2_{L^2_vH^{m_0}_x}.
\end{align*}
Taking $C_0>1$ sufficiently large, and then picking $\delta$ sufficiently small which depends on $C_0$, we get 
\begin{align*}
&C_0\partial_t\|t^{\kappa m}f\|_{L^2_vH^{m+m_0}_x}^2 
+ C_0\varepsilon\|Mt^{\kappa m}f\|_{L^2_vH^{m+m_0}_x}^2
+C_0\nu_0\|(a^{1/2})^wt^{\kappa m}f\|_{L^2_vH^{m+m_0}_x}^2+\|(b^{1/2})^wt^{\kappa m}f\|^2_{L^2_vH^{m+m_0}_x}\\
&\le  C_0C_{\delta}\big(\|f\|^2_{L^2_vH^{m_0}_x}+\|(a^{1/2})^wf\|^2_{L^2_vH^{m_0}_x}\big)-C\sum_{|\alpha|\le m}\partial_t\left(t^{\kappa m}\<D_x\>^{m_0}\partial^\alpha_xf,(t^{\kappa m}\theta^w\<y\>^{m_0}\widehat{\partial^\alpha_xf})^{\vee}\right)_{L^2}.
\end{align*}
Taking integral on $t\in(0,\tau)$, applying $\theta\in S(1)$ and \eqref{eq0}, we have 
\begin{align*}
&\quad\,C_0\|\tau^{\kappa m}f(\tau)\|_{L^2_vH^{m+m_0}_x}^2 
+ C_0\varepsilon\int^\tau_0\|Mt^{\kappa m}f\|_{L^2_vH^{m+m_0}_x}^2\,dt
\\
&\quad\,+C_0\nu_0\int^\tau_0\|(a^{1/2})^wt^{\kappa m}f\|_{L^2_vH^{m+m_0}_x}^2\,dt+\int^\tau_0\|t^{\kappa m}(b^{1/2})^wf\|^2_{L^2_vH^{m+m_0}_x}\,dt\\
&\le  C_0C_{\delta}\int^\tau_0\big(\|f\|^2_{L^2_vH^{m_0}_x}+\|(a^{1/2})^wf\|^2_{L^2_vH^{m_0}_x}\big)\,dt\\
&\quad-C\sum_{|\alpha|\le m}\left(\tau^{\kappa m}\<D_x\>^{m_0}\partial^\alpha_xf(\tau),(\tau^{\kappa m}\theta^w\<y\>^{m_0}\widehat{\partial^\alpha_xf(\tau)})^{\vee}\right)_{L^2}\\
&\le C_0C_{\delta}\varepsilon_0^2+C\|\tau^{\kappa m}f(\tau)\|_{L^2_vH^{m+m_0}_x}^2.
\end{align*}
Finally, notice that the second constant $C$ is independent of $\delta$, so we can pick $C_0$ sufficiently large to absorb the second term. Then for $T\in(0,1]$
\begin{align*}\notag
&\quad\,C_0\sup_{0\le t\le T}\|t^{\kappa m}f(t)\|_{L^2_vH^{m+m_0}_x}^2 
+ C_0\varepsilon\int^T_0\|Mt^{\kappa m}f\|_{L^2_vH^{m+m_0}_x}^2\,dt\\&\quad
+C_0\nu_0\int^T_0\|(a^{1/2})^wt^{\kappa m}f\|_{L^2_vH^{m+m_0}_x}^2\,dt+\int^T_0\|t^{\kappa m}(b^{1/2})^wf\|^2_{L^2_vH^{m+m_0}_x}\,dt\\
&\le C_0C\varepsilon_0^2.\notag
\end{align*}
\end{proof}

Finally, we can summarize the estimate on regularity for $f$ and complete the iteration to obtain the regularity for solution to Boltzmann equation. 

\begin{proof}[Proof of Theorem \ref{Main}]
	Let $m,k,l\ge 0$. Let $f$ be the solution to \eqref{lineareqvarepsilon}.
Estimate \eqref{30} and \eqref{24} gives that for $m,l\in\N$, 
\begin{align}\notag
	&\sup_{0\le t\le T}\|t^{\kappa m}f(t)\|_{L^2_vH^{m+m_0}_x}^2 
	+ \varepsilon\int^T_0\|Mt^{\kappa m}f\|_{L^2_vH^{m+m_0}_x}^2\,dt\\
	&\quad\quad
	+\sup_{0\le t\le T}\|\<v\>^lf\|_{L^2_vH^{m_0}_x}^2 
	+ \varepsilon\int^T_0\|M\<v\>^lf\|_{L^2_vH^{m_0}_x}^2\,dt
	\le C\varepsilon_0^2.\label{eq51a}
\end{align}
The constants are independent of $\varepsilon$. 
We integral \eqref{20} on $t$ and apply \eqref{eq51a}, then for $k\in\N$, 
\begin{align}
	\notag&\quad\,\sup_{0\le t\le T}\|t^{\kappa k}f\|_{H^k_vH^{m_0}_x}^2 
	+ \varepsilon\int^T_0\|Mt^{\kappa k}f\|_{H^k_vH^{m_0}_x}^2\,dt
	+\nu_0\int^T_0\|(a^{1/2})^wt^{\kappa k}f\|_{H^k_vH^{m_0}_x}^2\,dt\\
	&\le C_{\kappa,k}\big(\varepsilon \int^T_0\|Mf\|^2_{L^2_vH^{m_0}_x}\,dt + \int^T_0\|(a^{1/2})^wf\|^2_{L^2_vH^{m_0}_x}\,dt + \int^T_0\|t^{\kappa k}\<D_x\>^kf\|^2_{L^2_vH^{m_0}_x}\,dt\big)\notag\\
	&\le C_{\kappa,k}\varepsilon_0^2.\label{eq51}
\end{align}
Thus,
\begin{align}\label{eqq50}
	&\quad\,\sup_{0\le t\le T}\|t^{\kappa (m+k)}\<v\>^l\<D_v\>^k\<D_x\>^mf\|_{L^2_vH^{m_0}_x}^2 
	\notag\\
	&= \sup_{0\le t\le T}\big(t^{2\kappa m}\<D_x\>^{2m}\<v\>^{2l}f, t^{2\kappa k}\<v\>^{-2l}\<D_v\>^{k}\<v\>^{2l}\<D_v\>^{k}f\big)_{L^2_vH^{m_0}_x}\notag\\
	&\le C\big(\sup_{0\le t\le T}\|t^{2\kappa m}\<D_x\>^{2m}\<v\>^{2l}f\|_{L^2_vH^{m_0}_x}^2 +
	\sup_{0\le t\le T}\|t^{2\kappa k}f\|_{H^{2k}_vH^{m_0}_x}^2\big)\notag\\
	&\le C\big(\sup_{0\le t\le T}\|\<v\>^{4l}f\|_{L^2_vH^{m_0}_x}^2+\sup_{0\le t\le T}\|t^{4\kappa m}\<D_x\>^{4m}f\|_{L^2_vH^{m_0}_x}^2 +
	\sup_{0\le t\le T}\|t^{2\kappa k}f\|_{H^{2k}_vH^{m_0}_x}^2\big)\notag\\
	&\le C\varepsilon_0^2.
\end{align}

Noticing that the constants are independent of $\varepsilon$, by Banach-Alaoglu theorem, the solution $f=f_\varepsilon$ to equation \eqref{lineareqvarepsilon} weakly* converges to the weak solution $f$ to equation \eqref{lineareq} in the corresponding spaces as $\varepsilon\to 0$, which satisfies that for $m,k,l\in\N$, 
\begin{equation}\begin{split}
		\label{32}
		\sup_{0\le t\le T}\|t^{\kappa (m+k)}\<v\>^l\<D_v\>^k\<D_x\>^mf\|_{L^2_vH^{m_0}_x}^2 
		\le C\varepsilon_0^2,
	\end{split}
\end{equation}where the constant $C$ is independent of $T$. 
Now we $\varepsilon_0$ sufficiently small that the solution $f$ to \eqref{lineareq} satisfies
\begin{align}
	\sup_{0\le t\le T}\|t^{\kappa m}\<v\>^l\<D_v\>^k\<D_x\>^mf\|_{L^2_vH^{m_0}_x}^2 \le \varepsilon_1.
\end{align}This is the iteration assumption \eqref{assumong} and hence we can begin the iteration stated in section \ref{sec2}. 
Let $f^0 = 0$ and $f^{n+1}$ $(n\in\N)$ be the solution to 
\begin{align*}
	\partial_tf^{n+1} = Bf^{n+1} + \Gamma(f^n,f^{n+1}),\qquad f|_{t=0} = f_0.
\end{align*}
Then $f^{n}$ $(n\in\N)$ satisfies the iteration assumption \eqref{assumong}. Also the initial data satisfies \eqref{initial1}. Thus by iteration, the regularity estimate \eqref{32} gives that for $k,m,l,j\in\N$, 
\begin{equation}\begin{split}
		\sup_{0\le t\le T}\|t^{\kappa (m+k)}\<v\>^lf^{n+1}\|_{H^k_vH^{m+m_0}_x}^2
		\le C&\varepsilon_0^2,
	\end{split}
\end{equation}
Thus the approximation sequence $\{f^n(t)\}$ is bounded in corresponding spaces and hence by Banach-Alaoglu theorem, it has a weak* limit, which is exactly the solution $f$ to Boltzmann equation \eqref{Boltzmann} by the uniqueness of the solution to Boltzmann equation (cf. Theorem \ref{Thm1}). Also $f$ satisfies \eqref{32}.

But on the other hand, we have the uniform bound \eqref{globaletimate} that 
\begin{align*}
	\sup_{0\le t<\infty}\|f(t)\|_{L^2_vH^{m_0}_x}^2\le C\varepsilon^2_0. 
\end{align*}
This allows us to recover the regularity on any time interval $[T_1,T_2]$ $(T_2-T_1\le 1)$ (regarding $f(T_1)$ as the initial data and do the above calculation on $[T_1,T_2]$ instead of $[0,T]$), since the weak solution $f$ in Theorem \ref{Thm1} is unique and our analysis is independent of time $t$, i.e. the constant $C$ is independent of time $T\le1$.
So then we get the uniform bound on time $t$: for $\tau>0$, 
\begin{align*}
	\|\<v\>^l\<D_v\>^k\<D_x\>^mf\|^2_{L^\infty([\tau,\infty);L^2_vH^{m_0}_x)}\le C_\tau\varepsilon_0^2.
\end{align*}

By Sobolev embedding theorem, we have that the solution $f(t)$ to Boltzmann equation belongs to $C^\infty(\R^d_x;\S(\R^d_v))$ for $t\in (0,\infty)$. 
\end{proof}

\section{Appendix}
The function $a$ and $b$ can regenerate regularity on $v$ and $y$ respectively. That is for $K,J,k,l\ge 0$, we have 
\begin{equation}\label{eq09}\begin{split}
		a^{-K-J} &\le \<v\>^{-(\gamma+2s)K}\<v\>^{-\gamma J}\<\eta\>^{-2sJ}\\
		&\le \<v\>^{-l}\<\eta\>^{-k},\\
		b^{-K} &\le \<v\>^{-l}\<y\>^{-k},
	\end{split}
\end{equation}
for sufficiently large $K>>J>>1$. Notice that $\gamma$ may be negative in this paper, since $\gamma+2s>0$ is our only restriction. 

\paragraph{Pseudo-differential calculus}

We recall some notation and theorem of pseudo differential calculus. For details, one may refer to Chapter 2 in the book \cite{Lerner2010}, Proposition 1.1 in \cite{Bony1998-1999} and \cite{Beals1981,Bony1994} for details. Set $\Gamma=|dv|^2+|d\eta|^2$, but also note that the following are also valid for general admissible metric.
Let $M$ be an $\Gamma$-admissible weight function. That is, $M:\R^{2d}\to (0,+\infty)$ satisfies the following conditions:\\
(a). (slowly varying) there exists $\delta>0$ such that for any $X,Y\in\R^{2d}$, $|X-Y|\le \delta$ implies
\begin{align*}
	M(X)\approx M(Y);
\end{align*}
(b) (temperance) there exists $C>0$, $N\in\R$, such that for $X,Y\in \R^{2d}$,
\begin{align*}
	\frac{M(X)}{M(Y)}\le C\<X-Y\>^N.
\end{align*}
A direct result is that if $M_1,M_2$ are two $\Gamma$-admissible weight, then so is $M_1+M_2$ and $M_1M_2$. Consider symbols $a(v,\eta,\xi)$ as a function of $(v,\eta)$ with parameters $\xi$. We say that
$a\in S(\Gamma)=S(M,\Gamma)$ uniformly in $\xi$, if for $\alpha,\beta\in \N^d$, $v,\eta\in\Rd$,
\begin{align*}
	|\partial^\alpha_v\partial^\beta_\eta a(v,\eta,\xi)|\le C_{\alpha,\beta}M,
\end{align*}with $C_{\alpha,\beta}$ a constant depending only on $\alpha$ and $\beta$, but independent of $\xi$. The space $S(M,\Gamma)$ endowed with the seminorms
\begin{align*}
	\|a\|_{k;S(M,\Gamma)} = \max_{0\le|\alpha|+|\beta|\le k}\sup_{(v,\eta)\in\R^{2d}}
	|M(v,\eta)^{-1}\partial^\alpha_v\partial^\beta_\eta a(v,\eta,\xi)|,
\end{align*}becomes a Fr\'{e}chet space.
Sometimes we write $\partial_\eta a\in S(M,\Gamma)$ to mean that $\partial_{\eta_j} a\in S(M,\Gamma)$ $(1\le j\le d)$ equipped with the same seminorms.
We formally define the pseudo-differential operator by
\begin{align*}
	(op_ta)u(x)=\int_\Rd\int_\Rd e^{2\pi i (x-y)\cdot\xi}a((1-t)x+ty,\xi)u(y)\,dyd\xi,
\end{align*}for $t\in\R$, $f\in\S$.
In particular, denote $a(v,D_v)=op_0a$ to be the standard pseudo-differential operator and
$a^w(v,D_v)=op_{1/2}a$ to be the Weyl quantization of symbol $a$. We write $A\in Op(M,\Gamma)$ to represent that $A$ is a Weyl quantization with symbol belongs to class $S(M,\Gamma)$. One important property for Weyl quantization of a real-valued symbol is the self-adjoint on $L^2$ with domain $\S$. 

Let $a_1(v,\eta)\in S(M_1,\Gamma),a_2(v,\eta)\in S(M_2,\Gamma)$, then $a_1^wa_2^w=(a_1\#a_2)^w$, $a_1\#a_2\in S(M_1M_2,\Gamma)$ with
\begin{align*}
	a_1\#a_2(v,\eta)&=a_1(v,\eta)a_2(v,\eta)
	+\int^1_0(\partial_{\eta}a_1\#_\theta \partial_{v} a_2-\partial_{v} a_1\#_\theta \partial_{\eta} a_2)\,d\theta,\\
	g\#_\theta h(Y):&=\frac{2^{2d}}{\theta^{-2n}}\int_\Rd\int_\Rd e^{-\frac{4\pi i}{\theta}\sigma(X-Y_1)\cdot(X-Y_2)}(4\pi i)^{-1}\<\sigma\partial_{Y_1}, \partial_{Y_2}\>g(Y_1) h(Y_2)\,dY_1dY_2,
\end{align*}with $Y=(v,\eta)$, $\sigma=\begin{pmatrix}
	0&I\\-I&0
\end{pmatrix}$.
For any non-negative integer $k$, there exists $l,C$ independent of $\theta\in[0,1]$ such that
\begin{align}\label{sharp_theta}
	\|g\#_\theta h\|_{k;S(M_1M_2,\Gamma)}\le C\|g\|_{l,S(M_1,\Gamma)}\|h\|_{l,S(M_2,\Gamma)}.
\end{align}
Thus if $\partial_{\eta}a_1,\partial_{\eta}a_2\in S(M'_1,\Gamma)$ and $\partial_{v}a_1,\partial_{v}a_2\in S(M'_2,\Gamma)$, then $[a_1,a_2]\in S(M'_1M'_2,\Gamma)$, where $[\cdot,\cdot]$ is the commutator defined by $[A,B]:=AB-BA$.

For composition of pseudodifferential operator we have $a^wb^w= (a\#b)^w$ with 
\begin{align}\label{compostion}
	a\#b = ab + \frac{1}{4\pi i}\{a,b\} + \sum_{2\le k\le \nu}2^{-k}\sum_{|\alpha|+|\beta|=k}\frac{(-1)^{|\beta|}}{\alpha!\beta!}D^\alpha_\eta\partial^\beta_xaD^{\beta}_\eta\partial^\alpha_xb+r_\nu(a,b_),
\end{align}where $X=(v,\eta)$,
\begin{align*}
	r_\nu(a,b)(X) & = R_\nu(a(X)\otimes b(Y))|_{X=Y},\\
	R_\nu &= \int^1_0\frac{(1-\theta)^{\nu-1}}{(\nu-1)!}\exp\Big(\frac{\theta}{4\pi i}\<\sigma\partial_X,\partial_Y\Big)\,d\theta\Big(\frac{1}{4\pi i}\<\sigma\partial_X,\partial_Y\Big)^\nu.
\end{align*}

We can define a Hilbert space $H(M,\Gamma):=\{u\in\S':\|u\|_{H(M,\Gamma)}<\infty\}$, where
\begin{align}\label{sobolev_space}
	\|u\|_{H(M,\Gamma)}:=\int M(Y)^2\|\varphi^w_Yu\|^2_{L^2}|g_Y|^{1/2}\,dY<\infty,
\end{align}and $(\varphi_Y)_{Y\in\R^{2d}}$ is any uniformly confined family of symbols which is a partition of unity. If $a\in S(M)$ is a isomorphism from $H(M')$ to $H(M'M^{-1})$, then $(a^wu,a^wv)$ is an equivalent Hilbertian structure on $H(M)$. Moreover, the space $\S(\Rd)$ is dense in $H(M)$ and $H(1)=L^2$.

Let $a\in S(M,\Gamma)$, then
$a^w:H(M_1,\Gamma)\to H(M_1/M,\Gamma)$ is linear continuous, in the sense of unique bounded extension from $\S$ to $H(M_1,\Gamma)$.
Also the existence of $b\in S(M^{-1},\Gamma)$ such that $b\#a = a\#b = 1$ is equivalent to the invertibility of $a^w$ as an operator from $H(MM_1,\Gamma)$
onto $H(M_1,\Gamma)$ for some $\Gamma$-admissible weight function $M_1$.

For the metric $\Gamma=|dv|^2+|d\eta|^2$, the map $J^t=\exp(2\pi i D_v\cdot D_\eta)$ is an isomorphism of the Fr\'{e}chet space $S(M,\Gamma)$, with polynomial bounds in the real variable $t$, where $D_v=\partial_v/i$, $D_\eta=\partial_\eta/i$. Moreover, $a(x,D_v)=(J^{-1/2}a)^w$.

\paragraph{Carleman representation and cancellation lemma}

Now we have a short review of some useful facts in the theory of Boltzmann equation. One may refer to \cite{Alexandre2000,Global2019} for details. The first one is the so called Carleman representation. For measurable function $F(v,v_*,v',v'_*)$, if any sides of the following equation is well-defined, then
\begin{align}
	&\int_{\R^d}\int_{\mathbb{S}^{d-1}}b(\cos\theta)|v-v_*|^\gamma F(v,v_*,v',v'_*)\,d\sigma dv_*\notag\\
	&\quad=\int_{\R^d_h}\int_{E_{0,h}}\tilde{b}(\alpha,h)\1_{|\alpha|\ge|h|}\frac{|\alpha+h|^{\gamma+1+2s}}{|h|^{d+2s}}F(v,v+\alpha-h,v-h,v+\alpha)\,d\alpha dh,\label{Carleman}
\end{align}where $\tilde{b}(\alpha,h)$ is bounded from below and above by positive constants, and $\tilde{b}(\alpha,h)=\tilde{b}(|\alpha|,|h|)$, $E_{0,h}$ is the hyper-plane orthogonal to $h$ containing the origin. The second is the cancellation lemma. Consider a measurable function $G(|v-v_*|,|v-v'|)$, then for $f\in\S$,
\begin{align*}
	\int_{\R^d}\int_{\mathbb{S}^{d-1}}G(|v-v_*|,|v-v'|)b(\cos\theta)(f'_*-f_*)\,d\sigma dv_* = S*_{v_*}f(v),
\end{align*}where $S$ is defined by, for $z\in\R^d$,
\begin{align*}
	S(z)=2\pi \int^{\pi/2}_0 b(\cos\theta)\sin\theta\left(G(\frac{|z|}{\cos\theta/2},\frac{|z|\sin\theta/2}{\cos\theta/2}) - G(|z|,|z|\sin(\theta/2))\right)\,d\theta.
\end{align*}

\footnotesize
\bibliographystyle{plain}
\bibliography{1.bib}

\end{document}